\documentclass[11pt,reqno]{amsart}
\usepackage[T1]{fontenc}
\usepackage{graphicx}
\usepackage{cite}
\usepackage{empheq}
\usepackage{hyphenat}

\usepackage{color}
\definecolor{MyLinkColor}{rgb}{0,0,0.4}
\numberwithin{equation}{section}

\newcommand{\re}{\mathop{\rm Re}\nolimits}

\newcommand{\PV}{\mathop{\rm PV}\nolimits}

\newcommand{\0}{\Omega}
\newcommand{\e}{\varepsilon}

\newcommand{\p}{\partial}
\newcommand{\wt}{\widetilde}
\newcommand{\ov}{\overline}
\newcommand{\wh}{\widehat}

\newcommand{\oo}{\ov\omega}

\newcommand{\bA}{\mathbb{A}}
\newcommand{\bX}{\mathbb{X}}
\newcommand{\B}{\mathcal{B}}
\newcommand{\bB}{\mathbb{B}}

\newcommand{\cO}{\mathcal{O}}

\newcommand{\kF}{\mathcal{F}}

\newcommand{\kH}{\mathcal{H}}

\newcommand{\kL}{\mathcal{L}}

\newcommand{\C}{\mathbb{C}}

\newcommand{\R}{\mathbb{R}}

\newcommand{\N}{\mathbb{N}}

\DeclareMathOperator{\supp}{supp}
 
\newtheorem{thm}{Theorem}[section]
\newtheorem{prop}[thm]{Proposition}
\newtheorem{lemma}[thm]{Lemma}

\newtheorem{rem}[thm]{Remark}

\numberwithin{equation}{section}

\title[The Muskat problem in 2D]{Viscous displacement in porous media: \\the Muskat problem in 2D}

\author[B.--V. Matioc]{Bogdan--Vasile Matioc}
\address{Institut for applied mathematics, Leibniz University Hanover, Welfengarten~1, 30167 Hanover, Germany.}
\email{matioc@ifam.uni-hannover.de}


\usepackage[colorlinks=true,linkcolor=MyLinkColor,citecolor=MyLinkColor]{hyperref} 

\begin{document}
  
\begin{abstract}
We consider the Muskat problem describing the viscous displacement  in a two-phase fluid system located in an unbounded two-dimensional porous medium or Hele-Shaw cell. 
\mbox{After} formulating the mathematical model as an evolution problem for the sharp interface between the fluids, 
we show that Muskat problem with surface tension is a quasilinear parabolic problem, whereas, in the absence of surface tension effects, 
the Rayleigh-Taylor condition identifies a domain of parabolicity  for the fully nonlinear problem.
Based upon these aspects, we then establish  the local well-posedness for arbitrary large initial data in $H^s$, $s>2$, 
if surface tension is taken into account, respectively for arbitrary large initial data in $H^2$ that additionally satisfy
the Rayleigh-Taylor condition if surface tension effects are neglected. 
We also show that the problem exhibits the parabolic smoothing effect  and we provide criteria for the global existence of solutions. 
\end{abstract}

\maketitle

\tableofcontents

\section{Introduction and main results}\label{Sec0}
The Muskat problem is a model proposed by M. Muskat in \cite{Mu34} to describe the encroachment  of water into an oil sand.
This problem is related to the secondary phase of the oil extraction process  where  water   injection is sometimes used to increase the pressure in the oil reservoir and to drive the oil towards the extraction well. 
In this paper we consider an unbounded fluid system,  consisting of two immiscible and incompressible fluid phases,   which moves with constant speed $|V|\geq0$, either in a  horizontal or a vertical Hele-Shaw cell (or a homogeneous porous medium). 
Furthermore, we assume that the flows are two-dimensional and that the velocities are asymptotically  equal to $(0,V)$ far away from the origin.
In a reference frame which moves with the same speed as the fluids and in the same direction, 
the Muskat problem can be formulated as an evolution problem for the pair $(f,\ov\omega)$, where  $[y=f(t,x)+Vt]$ is a parametrization for the sharp interface that separates  the fluids, with $f$ asymptotically flat for large $x\in\R$, 
and $\ov\omega/\sqrt{1+f'^2}$ is the jump of the velocity at the free interface in 
tangential direction (see \eqref{vor}).
Mathematically,  we are confronted  with the following    evolution problem
\begin{subequations}\label{P}
\begin{equation}\label{P:1}
\left\{
\begin{array}{rlll}
 \displaystyle\p_tf(t,x)\!\!&=&\!\! \displaystyle\frac{1}{2\pi}\PV\int_\R\frac{y+f'(t,x)(f(t,x)-f(t,x-y))}{ y^2+(f(t,x)-f(t,x-y))^2 }\ov\omega(t,x-y)\, dy,\quad \text{$ t>0$, $x\in\R$},\\[2ex]
f(0)\!\!&=&\!\! f_0,
\end{array}
\right.
\end{equation} 
 where $f$ and $\ov\omega$  are additionally coupled  through the following relation
\begin{equation}\label{P:2}
\begin{aligned}
 &\Big[\sigma\kappa(f)-\Big(g(\rho_--\rho_+)  +\frac{\mu_--\mu_+}{k}V\Big)f\Big]'(t,x) \\[1ex]
 &\hspace{1cm} =\displaystyle\frac{\mu_-+\mu_+}{2k}\ov\omega(t,x) + \displaystyle\frac{\mu_--\mu_+}{2\pi k} \PV\int_\R\frac{ yf'(t,x)-(f(t,x)-f(t,x-y)) }{ y^2+(f(t,x)-f(t,x-y))^2 }\ov\omega(t,x-y)\, dy 
\end{aligned}
\end{equation}
for $ t>0$ and $x\in\R$.
\end{subequations}
We denote by  $(\,\cdot\,)'$ the spatial derivative  $\p_x,$ 
  $g$ is the Earth's gravity, $k$ is   the permeability of the homogeneous porous medium, $\sigma$ is the surface tension coefficient at the free boundary,  
  $\rho_\pm$ is the density and   $\mu_\pm$  the viscosity of the fluid located at $\0_\pm(t)$,  where
\[ \0_-(t):=[y<f(t,x)+Vt]\qquad\text{and}\qquad \0_+(t):=[y>f(t,x)+Vt].\]
Moreover,  $\kappa(f(t))$ is the curvature of the graph $[y=f(t,x)+tV]$ and
$\PV$ denotes the principal value which, depending on the regularity of the functions under the integral, is taken at zero and/or infinity.
If    $V$ is positive, then the fluid  $-$ expends into the region occupied by the fluid $+$  and vice versa, if $V$ is negative, then the fluid  $+$ expends into the region occupied by the fluid $-$ 
(see Section \ref{Sec1} for rigorous a derivation of \eqref{P}). 
When neglecting  surface tension effects we set $\sigma=0$ and we require that the first equation of \eqref{P:1} and the equation \eqref{P:2} hold also at $t=0$.

In the recent years the Muskat problem has received, due to its physical relevance, much attention especially in the field of applied mathematics.
In the absence of surface tension effects the local  existence of solutions has been first addressed by F. Yi in \cite{Y96} under the assumption that the Rayleigh-Taylor condition holds.
The Rayleigh-Taylor condition \cite{ST58} is a sign restriction on the jump of the pressure gradients in normal direction at the interface  $[y=f_0(x)]$, and it reads
\begin{align}\label{RT}
 \p_\nu p_-< \p_\nu p_+ \qquad\text{on $[y=f_0(x)]$,}
\end{align}
where  $p_\pm$ is the pressure of the fluid $\pm$ and  $\nu$ the outward normal at $[y=f_0(x)]$ with respect to $\0_-(0)$.
Thereafter, questions related to the well-posedness of the Muskat problem  and other qualitative aspects of the  dynamics have been studied in
\cite{A04, BCG14, BCS16, BS16x, CGSV15x, CCG11, CCG13b, CG07, CG10, CGO14, CGCSS14x, SCH04, BV14, EMM12a, EM11a,  EMW15, M16x, CCGS13, GB14, CCFG13, CCFGL12, GG14, CGFL11,    GS14} in several physical scenarios and with 
various  methods.
These references show the Rayleigh-Taylor condition is crucial in the analysis of this problem.   
In the regime where the Rayleigh-Taylor condition  holds with reverse inequality sign, for example if a less viscous fluid displaces a more viscous one, or when a more dense fluid sits on top of a less dense one, physical experiments evidence  
the occurrence of  viscous fingering, cf. \cite{H87, ST58}, and the Muskat problem is ill-posed, cf. e.g. \cite{SCH04, EM11a, CG07}. 
On the other hand, it  was recently shown in \cite{EMW15}, in a  bounded and periodic striplike geometry, that the Rayleigh-Taylor condition actually
identifies   a domain of parabolicity for the Muskat problem.

When surface tension effects are taken into consideration, it was proven in \cite{EMW15, PS16, PS16x}, in bounded geometries, that the Muskat problem is a quasilinear parabolic problem for arbitrary large initial data, 
without  any kind of restrictions.
Also in this setting, the solvability of the problem has been addressed in several physical scenarios with quite intricate methods \cite{A14, FT03, HTY97,EMM12a, EM11a, T16}.

The first goal of this paper is to    prove that the classical formulation of the Muskat problem, see Section \ref{Sec1}, is equivalent to  the evolution problem \eqref{P}, cf. Proposition~\ref{PE}. 

Our second goal   is to extend the methods that have been recently applied in \cite{M16x}, in the particular case of fluids with equal viscosities, 
to the general case considered herein in order to establish the 
local well-posedness for the Muskat problem with and without surface tension by similar strategies and in a very general context.
If the fluids have equal viscosities, the equation  \eqref{P:2}  determines $\oo$ as a function of $f$,  and \eqref{P} becomes an evolution problem for $f$ only. 
Surprisingly, the analysis in  \cite{M16x} shows that the corresponding evolution problem is of  quasilinear parabolic type
in both regimes, that is for $\sigma>0$, or when $\sigma=0$ and the Rayleigh-Taylor condition holds.
However, for $\mu_-\neq \mu_+$, the equation \eqref{P:2} is  implicit and this fact enhances the nonlinear and nonlocal character of the problem and makes the analysis more involved.

In the case when $\sigma=0$ and the Rayleigh-Taylor condition holds, the  well-posedness of the problem is still an open question. 
Local existence of solutions to \eqref{P} has been first addressed  in \cite{CCG11} for arbitrary large  data in $H^3$, 
and in three space dimensions in \cite{CCG13b} for initial data in $H^4$. 
Global existence is established in  \cite{SCH04}  in the periodic case and for small initial data.
Quite recently, the authors of \cite{BCS16} have proven the existence and uniqueness of solutions which satisfy an additional energy estimate  for initial data in $H^2$ which are small with respect to some $H^{3/2+\e}$-norm.   
In Theorem~\ref{MT2} we show that the Muskat problem without surface tension is well-posed for arbitrary large initial data in $H^2$.
To achieve this result we formulate \eqref{P} as a fully nonlinear  evolution problem for $f$ and we prove that the set of initial data for which the Rayleigh-Taylor condition holds defines, also in this geometry,
a domain of parabolicity for the Muskat problem. 
It is worth emphasizing that the quasilinear character, present for $\mu_-=\mu_+$, is not preserved when $\mu_-\neq\mu_+$ and this makes the Muskat problem without surface tension more difficult to handle.

For $\sigma>0$, the local well-posedness of \eqref{P} has been addressed in \cite{A14} for initial data in $H^s$, with $s\geq6$ (see also \cite{T16} for a global existence result for small initial data in $H^s$, with $s\geq6$).
Exploiting the quasilinear structure of   the curvature term,  we show that in this regime \eqref{P} can be formulated  as a quasilinear parabolic evolution problem.
This property enables us to establish the local 
well-posedness of \eqref{P} for arbitrary large initial in $H^s$, with $s>2$, cf. Theorem~\ref{MT1}. 
In particular, we may chose the initial data such that the curvature is unbounded or discontinuous.

Moreover, we show that the Muskat problem features the effect of parabolic smoothing: solutions (which possess additional regularity when $\sigma=0$) become instantly real-analytic in the time-space domain.
Besides, we provide criteria for the global existence of solutions.

The first main result of this paper is the following theorem.
\begin{thm}[Well-posedness: with surface tension]\label{MT1}
Let $\sigma>0$.
The problem \eqref{P} possesses  for each   $f_0\in H^s(\R)$,  $s\in(2,3),$ a unique maximal solution
\[ f:=f(\,\cdot\,; f_0)\in C([0,T_+(f_0)),H^s(\R))\cap C((0,T_+(f_0)), H^3(\R))\cap C^1((0,T_+(f_0)), L_2(\R)), \] 
with $T_+(f_0)\in(0,\infty]$, and $[(t,f_0)\mapsto f(t;f_0)]$ defines a  semiflow on $H^s(\R)$. Additionally, if   
\[
\sup_{[0,T_+(f_0))\cap[0,T]}\|f(t)\|_{H^s}<\infty\qquad\text{for all $T>0$},
\]
then  $T_+(f_0)=\infty$.
 Moreover, given $k\in\N$, it holds that
 \[
 f\in C^\omega((0,T_+(f_0))\times\R,\R)\cap C^\omega ((0,T_+(f_0)), H^k(\R))\footnote{Here and in the following  $C^\omega$ stands  for real-analyticity, while $C^{1-}$ denotes local Lipschitz continuity.}.
 \] 
  In particular, $f(t,\,\cdot\,)$ is real-analytic for each $t\in(0,T_+(f_0)).$
 \end{thm}
 
 We emphasize that exactly the same result as in Theorem \ref{MT1} has been achieved in \cite{M16x} in the simpler case of fluids with equal viscosities.

  When surface tension is neglected, that is $\sigma=0$, we assume that 
  \begin{align}\label{TET}
   \Theta:=g(\rho_--\rho_+)  +\frac{\mu_--\mu_+}{k}V\neq0.
  \end{align}
  The situation when $\sigma=0=\Theta$  is special, because in this case the problem \eqref{P} possesses   for each $f_0\in H^s(\R)$, with $s>3/2$,  a unique global solution 
 $f(t):=f_0$ for all $t\in\R$, cf. Section \ref{Sec4}. The  corresponding flow is stationary with constant  velocities equal to $(0,V)$  and hydrostatic pressures.
 
In order to discuss the well-posedness of \eqref{P} with $\sigma=0\neq\Theta$, we introduce the set of initial data for which the Rayleigh-Taylor condition holds as
 \[
 \cO:=\{f_0\in H^2(\R)\,:\, \text{$\p_\nu p_-< \p_\nu p_+$ on $[y=f_0(x)]$}\}.
 \]
 The Rayleigh-Taylor condition is reformulated later on, cf. \eqref{RTT2}, where it is also proven that $\cO$ is an open subset of $H^2(\R)$. 
 Our analysis in Section \ref{Sec4} shows that $\cO$ is nonempty if and only if 
 \begin{align}\label{NEQ}
 \Theta >0.
 \end{align}
The relation \eqref{NEQ} is the classical condition found within the linear theory by Saffman and Taylor \cite{ST58}.  
In particular, if the flow takes place in a vertical Hele-Shaw cell and $V=0$, then the less dense fluid lies above.
For flows in   horizontal Hele-Shaw cells the effects due to gravity  are usually  neglected, that is $g=0$, and  \eqref{RT} implies
that  $V\neq0$ and that the more viscous fluid expends  into the region occupied by
the less viscous one.

We  now come to the second main result of this paper.
 \begin{thm}[Well-posedness: without surface tension]\label{MT2}
Let $\sigma=0$, $\mu_-\neq\mu_+$\footnote{Theorem~\ref{MT2} is still valid when $\mu_-=\mu_+,$ however in this case its assertions can be improved, cf.  \cite[Theorem~1.1]{M16x}.}, and assume that \eqref{NEQ} holds.
Given $f_0\in\cO$, the  problem \eqref{P} possesses  a   solution
\[ f\in C([0,T],\cO)\cap C^1([0,T], H^1(\R))\cap C^{\alpha}_{\alpha}((0,T], H^2(\R)) \] 
for some $T>0$ and an arbitrary  $\alpha\in(0,1)$.
It further holds:
\begin{itemize}
 \item[$(i)$] $f$ is the unique solution to \eqref{P} belonging to 
 \[\bigcup_{\beta\in(0,1)} C([0,T],\cO)\cap C^1([0,T], H^1(\R))\cap C^{\beta}_{\beta}((0,T], H^2(\R));\]
 \item[$(ii)$] $f$ may be extended to a maximally defined solution 
 $$f(\,\cdot\,; f_0)\in C([0,T_+(f_0)),\cO)\cap C^1([0,T_+(f_0)), H^1(\R))\cap \bigcap_{\beta\in(0,1)}C^{\beta}_{\beta}((0,T], H^2(\R))$$
 for all $T<T_+(f_0)$, where $T_+(f_0)\in (0,\infty];$
 \item[$(iii)$] The solution map   $[(t,f_0)\mapsto f(t;f_0)]$ defines a  semiflow on $\cO$ which is real-analytic in the open set $\{(t,f_0)\,:\, f_0\in\cO,\, 0<t<T_+(f_0)\}$;
 \item[$(iv)$] If $f(\,\cdot\,; f_0):[0,T_+(f_0))\cap[0,T]\to\cO$ is uniformly continuous for all $T>0$, then either
 \[\text{ $T_+(f_0)<\infty$ and $\lim_{t\to T_+(f_0)} f(t;f_0)\in\p\cO,$ or $T_+(f_0)=\infty;$}\]
 \item[$(v)$] If $f(\,\cdot\,; f_0)\in B((0,T), H^{2+\e}(\R))$ for some $T\in(0,T_+(f_0))$ and $\e\in(0,1)$, then 
 \[
 f\in C^\omega((0,T)\times\R,\R)\cap C^\omega ((0,T), H^k(\R))\qquad\text{for each $k\in\N.$}
 \]  
\end{itemize}
 \end{thm}
 Given $T>0$  and a Banach space $\bX$, we  let $B((0,T], \bX)$ [resp. $B((0,T), \bX)$] denote the Banach space of all bounded functions form $(0,T]$ [resp. $(0,T)$] into $\bX$, and, given  $\alpha\in(0,1),$ we set 
 \[
 C^\alpha_\alpha((0,T], \bX):=\Big\{u\in B((0,T], \bX)\,:\, \sup_{s\neq t}\frac{|t^\alpha u(t)-s^\alpha u(s)|}{|t-s|^\alpha}<\infty\Big\}.
 \]
 
With respect to $(iv)$  we add the following comments. Firstly,  as shown in \cite[Theorem~1.1]{CCFGL12} in the case when $\mu_-=\mu_+$, there exist solutions which are not uniformly continuous in $\cO$, 
in the sense that their slope blows up in finite time.
Secondly, there exist global solutions to \eqref{P}, cf.  \cite[Theorem~3.1]{CCGS13} (see also \cite[Corollary~1.4]{M16x}) or \cite[Theorem 2.2]{BCS16} (in the periodic setting).
Lastly, the existence of solutions which are uniformly bounded in $H^2(\R)$ and violate the Rayleigh-Taylor sign condition at time $T_+(f_0)<\infty$   is, to the best of our knowledge,  an open issue.

 The condition that $f\in B((0,T), H^{2+\e}(\R))$ for some $T\in(0,T_+(f_0))$  imposed at $(v)$ is a technical assumption. 
 Nevertheless, if $f_0\in \cO\cap H^3(\R),$   our arguments can be extended to show  that  Theorem~\ref{MT2} still  holds true if we replace $\cO$ by $\cO \cap H^3(\R) $ and $H^k(\R)$ by $H^{k+1}(\R)$ for $k\in\{1,2\},$ 
 possibly with a smaller  maximal existence time  $T_{+,3}(f_0)$. Hence, for solutions that start in $H^3(\R)$, the property required at $(v)$ is satisfied for all $T<T_{+,3}(f_0)$ and all $\e\in(0,1).$
 This additional regularity is needed for our argument because the uniqueness property in Theorem~\ref{MT2} holds  only for solutions that additionally
 belong to the space $C^{\alpha}_{\alpha}((0,T], H^2(\R)) $, for some $\alpha\in(0,1)$, and this space is not
 sufficiently flexible  with respect to the parameter  trick used   in  the proof of Theorem~\ref{MT2}.

\section{The governing equations and the equivalence result}\label{Sec1}
We start by presenting the classical formulation of the Muskat problem introduced  in Section \ref{Sec0}.
First of all, both fluids are taken to be incompressible, immiscible, and of Newtonian type. 
Since  flows in porous media or Hele-Shaw cells occur at low Reynolds numbers, they are usually modeled as being two-dimensional and  Darcy's law is used instead of the conservation of momentum equation \cite{Be88}.
Hence, the equations  of motion in the fluid layers are\footnote{The first equation of \eqref{eq:S1} expresses the fact that ${\rm div}\,  v_+(t)=0$ in $ \Omega_+(t)$ and that ${\rm div}\,  v_-(t)=0$ in $ \Omega_-(t).$
This convention is also use in the second equation of \eqref{eq:S1} and at many other places in the paper in various contexts.}
\begin{subequations}\label{PB}
\begin{equation}\label{eq:S1}
\left\{\begin{array}{rllllll}
{\rm div}\,  v_\pm(t)\!\!&=&\!\!0&\text{in $ \Omega_\pm(t) , $}\\[1ex]
v_\pm(t)\!\!&=&\!\!-\cfrac{k}{\mu_\pm}\big(\nabla p_\pm(t)+(0,\rho_\pm g)\big)&\text{in $ \0_\pm(t)$},
\end{array}
\right.
\end{equation} 
with $v_\pm:=(v_\pm^1,v_\pm^2)$ denoting the velocity of the fluid $\pm.$ 
These equations are supplemented by  the natural  boundary conditions on the free surface
\begin{equation}\label{eq:S2}
\left\{\begin{array}{rllllll}
 \langle v_+(t)| \nu(t)\rangle\!\!&=&\!\! \langle v_-(t)| \nu(t)\rangle &\text{on $ [y=f(t,x)+Vt],$}\\[1ex]
p_+(t)-p_-(t)\!\!&=&\!\!\sigma\kappa(f(t))&\text{on $ [y=f(t,x)+Vt],$} 
\end{array}
\right.
\end{equation} 
where  $\nu(t) $ is the unit normal at $[y=f(t,x)+Vt]$ pointing into $\0_+(t)$ and   $\langle\,\cdot\,|\,\cdot\,\rangle$  the Euclidean inner product  on $\R^2$.
Furthermore, we impose the  following far-field boundary  conditions
 \begin{equation}\label{eq:S3}
 \left\{\begin{array}{llllll}
f(t,x)\to0 &\text{for  $|x|\to\infty$,}\\[1ex]
v_\pm(t,x,y)\to (0,V) &\text{for  $|(x,y)|\to\infty$}.
\end{array}
\right.
\end{equation}  
The motion of the interface $[y=f(t,x)+Vt]$ is coupled to that of the fluids through  the kinematic boundary condition
 \begin{equation}\label{eq:S4}
 \p_tf(t)\, =\,\langle v_\pm(t)| (-f'(t),1)\rangle-V\qquad\text{on    $ [y=f(t,x)+ Vt],$}
\end{equation}  
and the interface at time $t=0$ is  assumed to be known
\begin{equation}\label{eq:S5}
f(0)\, =\, f_0.
\end{equation} 
\end{subequations}

 We now rewrite   the classical formulation \eqref{PB} of the Muskat problem   in a coordinates system which moves with the same speed and in the same direction as the fluid system. 
To this end  we introduce
\begin{equation*}
 \left\{
 \begin{aligned}
 & \wt v_\pm (t,x,y):=v_\pm(t,x,y+Vt)-(0,V),\\[1ex]
  &\wt p_\pm (t,x,y):=p_\pm(t,x,y+Vt)
 \end{aligned}
 \right. \qquad\text{in $\0_{\pm}^0(t):=\0_\pm(t)-(0,Vt)$.}
\end{equation*}
It is not difficult to see that the equations \eqref{PB} are equivalent to the following   system of equations which has $(f, \wt v_\pm, \wt p_\pm)$ as unknowns
\begin{equation}\label{EPP}
\left\{
\begin{array}{rlllll}
 {\rm div}\,  \wt v_\pm(t)\!\!&=&\!\! 0 &&\text{in $ \0_{\pm}^0(t),$}\\[1ex]
\wt v_\pm(t)\!\!&=&\!\! -(0,V)-{k}\mu_\pm^{-1}\big(\nabla \wt p_\pm(t)+(0,\rho_\pm g)\big)&&\text{in $ \0_{\pm}^0(t),$} \\[1ex]
\langle \wt v_+(t)| \nu(t)\rangle\!\!&=&\!\! \langle \wt v_-(t)| \nu(t)\rangle&& \text{on $ [y=f(t,x)], $}\\[1ex]
\wt p_+(t)-\wt p_-(t)\!\!&=&\!\! \sigma\kappa(f(t))&&\text{on $ [y=f(t,x)], $}\\[1ex]
f(t,x)&\to &0 &&\text{for  $|x|\to\infty$, }\\[1ex]
\wt v_\pm(t,x,y)&\to& 0 &&\text{for  $|(x,y)|\to\infty,$}\\[1ex]
\p_tf(t) \!\!&=&\!\!\langle \wt v_\pm(t)| (-f'(t),1)\rangle && \text{on   $ [y=f(t,x)], $}\\[1ex]
f(0)\!\!&=&\!\!f_0. 
\end{array}
\right.
\end{equation}

Before stating the equivalence result, cf. Proposition~\ref{PE}, we first give a preparatory lemma, which is needed in the proof of Proposition~\ref{PE} and also later on in the analysis (see the proof of Theorem~\ref{T:I1}).
The proof of Lemma~\ref{L:A2} is based on classical arguments used to establish the  Plemelj formula and the Privalov theorem for  Cauchy-type integrals defined on regular curves, 
see e.g. \cite{JKL93}, and on the Lemmas~\ref{L:99a}-\ref{L:99d}.
Details of the proof  can be found, in a particular case, in \cite[Lemma~A.2.]{M16x}.  
\begin{lemma}\label{L:A2}Given $f\in H^2(\R)$ and $\ov\omega\in H^1(\R)$, we define
 \begin{equation}\label{V1'}
 \wh v(x,y):=\frac{1}{2\pi} \int_\R\frac{(-(y-f(s),x-s)}{ (x-s)^2+(y-f(s))^2 }\ov\omega(s)\, ds\qquad \text{in $\R^2\setminus[y=f(x)]$.}
\end{equation}
Let further  $\0_-^0:=[y<f(x)]$, $ \0_+^0:=[y>f(x)]$, and $\wh v_\pm:=\wh v\big|_{\0_\pm^0}$. Then,     $\wh v_\pm\in C (\ov{\0_\pm^0})\cap C^1({\0_\pm^0})$ and
 \begin{equation}\label{LA2}
\wh v_\pm(x,y)\to0 \qquad\text{for \, $|(x,y)|\to\infty$}.
\end{equation} 
Additionally, if $\ov\omega \in C^\infty_0(\R),$ then there exists a positive integer $N\in\N$  and a constant $C$ such that
 \begin{equation}\label{LA3}
|\wh v_\pm(x,y)|\leq\frac{ C\|\ov\omega\|_1}{|(x,y)|} \qquad\text{for all  $(x,y)\in\0_\pm^0$ with $|(x,y)|\geq N$}.
\end{equation}
\end{lemma}
\begin{proof}
 The first two claims can be established  in the same way as in \cite[Lemma~A.2]{M16x}, while \eqref{LA3} is a simple exercise.
\end{proof}

In the  particular case when  $\mu_-=\mu_+,$   \eqref{P:2} gives a precise correlation between the smoothness of  $\oo$ and that  of $f$. 
This  correlation is for $\mu_-\neq\mu_+$ no longer obvious.
We prove herein, cf. Proposition~\ref{P:I2}, that for $f\in H^2(\R)$, the equation \eqref{P:2} has a unique solution $\oo\in H^1(\R)$, provided that the left-hand side of \eqref{P:2} belongs to $H^1(\R)$. 
If $\sigma>0$, the latter requirement   implies that in fact $f\in H^4(\R)$ is needed. Thanks to the parabolic smoothing in Theorem~\ref{MT2}, this additional regularity is inherited by all solutions.  
This is one of the reasons, besides the difference in nonlinear behavior,  why we separate in Proposition~\ref{PE}   the cases   $\sigma=0$ and $\sigma>0$.

\begin{prop}[Equivalence of the two formulations]\label{PE}
Let   $T\in(0,\infty] $ be given.
\begin{itemize}
 \item[$(a)$] Let $\sigma=0$. 
The following are equivalent:\\[-1.5ex]
 \begin{itemize}
 \item[$(i)$] the Muskat problem \eqref{PB} for $ f\in   C^1([0,T), L_2(\R))$ and
 \begin{align*}
  \bullet &\, \, \, \, f(t)\in  H^2(\R),\, \, \ov\omega(t):=\big\langle (v_-(t)-v_+(t))|_{[y=f(t,x)+Vt]} \big|(1,f'(t))\big\rangle \in  H^1(\R),\\[1ex]
  \bullet &\, \, \, v_\pm(t)\in C(\ov{\0_\pm(t)})\cap C^1({\0_\pm(t)}), \, p_\pm(t)\in C^1(\ov{\0_\pm(t)})\cap C^2({\0_{\pm}(t)}) 
 \end{align*}
 for all $t\in[0,T)$;\\[-2ex]
\item[$(ii)$] the evolution  problem \eqref{P} for $f\in   C^1([0,T), L_2(\R))$, $f(t)\in  H^2(\R)$, and $  \ov\omega(t)\in  H^1(\R)$ for all $t\in[0,T)$.\\[-1.5ex]
\end{itemize}
\item[$(b)$] Let $\sigma>0$. 
The following are equivalent:\\[-1.5ex]
 \begin{itemize}
 \item[$(i)$] the Muskat problem \eqref{PB} for $f\in   C^1((0,T), L_2(\R))\cap C([0,T), L_2(\R))$ and
 \begin{align*}
  \bullet &\, \, \,  f(t)\in  H^4(\R),   \,\,\ov\omega(t):=\big\langle (v_-(t)-v_+(t))|_{[y=f(t,x)+Vt]} \big|(1,f'(t))\big\rangle \in  H^1(\R),  \\[1ex]
  \bullet &\, \, \, v_\pm(t)\in C(\ov{\0_\pm(t)})\cap C^1({\0_\pm(t)}), \, p_\pm(t)\in C^1(\ov{\0_\pm(t)})\cap C^2({\0_{\pm}(t)}) 
 \end{align*}
 for all $t\in(0,T);$\\[-2ex]
\item[$(ii)$] the evolution problem \eqref{P} for $f\in   C^1((0,T), L_2(\R))\cap C([0,T), L_2(\R))$, $ f(t)\in  H^4(\R)$, and $ \ov\omega(t)\in  H^1(\R) $ for all $t\in(0,T)$.
\end{itemize}
\end{itemize}
\end{prop}
\begin{proof}
We only  prove the claim for $\sigma=0$ (the proof of $(b)$ is similar). 
We first  consider  the implication $(i)\Rightarrow (ii)$. Given a set $E,$ we denote herein by ${\bf 1}_E$ the characteristic function of $E$.  
Assume that $(f,v_\pm,p_\pm)$ is a solution to \eqref{PB}  on $[0,T)$ and let  $t\in[0,T)$ be fixed (the time dependence is not written explicitly in  this proof).
It is more convenient   to work here with the formulation \eqref{EPP}. 
 Stokes' theorem   and the second equation of \eqref{eq:S1}  show  that the vorticity  $\omega:={\rm rot\,} \wt v:=\p_x\wt v^2-\p_y\wt v^1 $ defined by the global  velocity field   
 $\wt v:=(\wt v^1,\wt v^2):=\wt v_-{\bf 1}_{[y\leq f(x)]}+\wt v_+{\bf 1}_{[y> f(x)]}$
 is supported on the free boundary, that is 
\[\langle \omega,\varphi\rangle=  \int_\R\ov\omega(x)\varphi(x,f(x))\, dx\qquad\text{for all $\varphi\in C^\infty_0(\R^2)$,}\]
where
\begin{align}\label{vor}
 \ov\omega:= \big\langle(\wt v_- -\wt v_+)|_{[y=f(x)]}\big| (1,f')\big\rangle.
\end{align}

We now claim that the velocity is given by the Biot-Savart law, that is $\wt v=\wh v$ in $\R^2\setminus [y=f(x)]$,
where $\wh v$ is defined in \eqref{V1'} and $\ov\omega$  in \eqref{vor}.
Indeed, according to Plemelj formula, cf. e.g. \cite{JKL93},  the limits  $\wh v_-(x,f(x))$ and $\wh v_+(x,f(x))$  of $\wh v$ at $(x,f(x))$ when we approach this point
from above the interface $[y=f(x)] $ or from below, respectively, are
\begin{align}
\wh v_\pm(x,f(x))=\frac{1}{2\pi}\PV\int_\R\frac{(-(f(x)-f(x-s)),s)}{ s^2+(f(x)-f(x-s))^2 }\ov\omega(x-s)\, ds \mp\frac{1}{2}\frac{(1,f'(x))\ov\omega(x)}{1+{f'}^2(x)},\qquad x\in\R.\label{V2}
\end{align}
Moreover,   the restrictions $\wh v_\pm$ of $\wh v$ to $\0_\pm^0$ belong to $ C(\ov{\0_\pm^0})\cap C^1({\0_\pm^0})$, they satisfy the first, third, sixth equation of \eqref{EPP},   and 
${\rm rot\,}\wh v_\pm=\p_x\wh v_\pm^2-\p_y\wh v_\pm^1=0$ in $\0_\pm^0$.
We now introduce  $V_\pm:=\wt v_\pm-\wh v_\pm$, we set  $V:=(V^1,V^2):=V_-{\bf 1}_{[y\leq f(x)]}+V_+{\bf 1}_{[y> f(x)]}$, and  we consider the stream functions
\[
\psi_\pm(x,y):=\int^{y}_{f(x)}V_\pm ^1(x,s)\, ds-\int_0^x  \langle V_\pm(s,f(s)) | (-f'(s),1)\rangle\, ds\qquad\text{for $(x,y)\in\ov \0_\pm^0.$}
\]
The properties of $\wh v_\pm$ established above together with \eqref{V2} and Stokes' theorem show that the function  
$\psi:=\psi_-{\bf 1}_{[y\leq f]}+\psi_+{\bf 1}_{[y>f]}$ satisfies $\psi\in C(\R^2)$ and $\Delta\psi=0$ in $\mathcal{D}'(\R^2)$. 
Hence, $\psi$ is the real part of a holomorphic function $u:\C\to\C$.
Since $u'$ is also holomorphic and $u'=-(V^2,V^1)$ is bounded and vanishes for $|(x,y)|\to\infty$, it follows that $u'=0$, hence $\wt v_\pm=\wh v_\pm$. 
Differentiating  now the fourth equation of \eqref{EPP} once, the second equation of \eqref{EPP} and \eqref{V2} lead us to
\begin{align*}
&\hspace{-1cm}\Big[\sigma\kappa(f)-\Big(g(\rho_--\rho_+)+\frac{\mu_--\mu_+}{k}V\Big) f\Big]'(x)\\[1ex]
&=\frac{\mu_-+\mu_+}{2k}\ov\omega(x)+ \frac{\mu_--\mu_+}{2\pi k}\PV\int_\R\frac{f'(x)s-(f(x)-f(x-s))}{ s^2+(f(x)-f(x-s))^2 }\ov\omega(x-s)\, ds 
\end{align*}
for all $x\in\R$.
Finally, in view of \eqref{V2} and of the seventh equation of \eqref{EPP}, we may conclude that $(f,\ov w)$ is a solution to  \eqref{P}.

For the inverse implication  we define $\wt v_\pm\in C(\ov{\0_\pm^0})\cap C^1({\0_\pm^0})$ according to \eqref{V1'} and  the pressures $\wt p_\pm\in C^1(\ov\0_\pm^0)\cap C^2(\0_\pm^0)$  by the formula
\begin{align*} 
\wt p_\pm(x,y):=c_\pm-\frac{\mu_\pm}{k}\int_0^x \wt v_\pm^1(s,\pm d)\, ds-\frac{\mu_\pm}{k}\int_{\pm d}^y \wt v_\pm^2(x,s)\, ds-\Big(\rho_\pm g+\frac{\mu_\pm V}{k}\Big)y, \qquad (x,y)\in\ov\0_\pm,
\end{align*}
where $d$ is a positive constant satisfying  $d>\|f\|_ \infty$ and $c_\pm\in\R$. 
For a proper choice of $c_\pm$,  the tuple $(f,\wt p_\pm,\wt v_\pm)$ solves all the equations of \eqref{EPP} and  possesses the regularity properties states at $(i)$.
This  completes the proof of $(a)$.
\end{proof}

\section{On the resolvent set of the adjoint of the double layer potential}\label{Sec2}

In order to solve the Muskat problem \eqref{P}, with and without surface tension, we basically follow the same strategy. The  first step in our approach is to formulate  the system \eqref{P} as an evolution problem for $f$. 
To this end, we have to address the solvability of the equation \eqref{P:2}, which is the content of this section.
This issue is equivalent to   inverting the linear operator   $(1+a_\mu\bA(f))$, where
\begin{align}\label{DL1}
 \bA(f)[\ov\omega](x):=\frac{1}{\pi}\PV\int_\R\frac{ yf'(x)-(f(x)-f(x-y)) }{ y^2+(f(x)-f(x-y))^2 }\ov\omega(x-y)\, dy, 
\end{align}
and where
\[
a_\mu:= \frac{\mu_--\mu_+}{ \mu_-+\mu_+}
\]
denotes the Atwood number.
The operator $\bA(f)$  can be viewed as  the adjoint of the double layer potential, see e.g. \cite{FJR78, Ve84} and Lemma \ref{L:99f}. 
The resolvent set of $\bA(f)$ has been studied previously 
 in the literature (see \cite{FJR78, Ve84, CM97, CCG11, CCG13b} and the references therein), but mostly in bounded geometries where $\bA(f)$ 
 is a compact operator. 
 With respect to our functional analytic approach to \eqref{P}, the existing results cannot be applied, especially because the invertibility  in $\kL(H^1(\R))$ is established for functions  $f$ that are to regular.
For this reason we readdress this issue below, the emphasis being on finding the optimal correlation between the regularity of $f$ and the order of the  Sobolev space where the invertibility is considered, see Remark~\ref{R:IR}.
It is important to note, in the context of the Muskat problem \eqref{P}, that the Atwood number satisfies $|a_\mu|<1.$
  \medskip

\paragraph{\bf Some  multilinear   integral operators}
We now introduce a class of multilinear singular  operators which we encounter  later on when solving the implicit equation \eqref{P:2} for $\oo$.
Given $n,m\in\N$, with $m\geq1, $   we define the  singular integral operator 
\[
B_{n,m}(a_1,\ldots, a_{ m})[b_1,\ldots, b_n,\oo](x):=\PV\int_\R  \frac{\oo(x-y)}{y}\cfrac{\prod_{i=1}^{n}\big(\delta_{[x,y]} b_i /y\big)}{\prod_{i=1}^{m}\big[1+\big(\delta_{[x,y]}  a_i /y\big)^2\big]}\, dy,
\]
where   $a_1,\ldots, a_{m},b_1,\ldots, b_n:\R\to\R$  are Lipschitz functions and $\oo\in L_2(\R)$.
To keep the formulas short, we have set
 \begin{align*}
   \delta_{[x,y]}a:=a(x)-a(x-y)  \qquad \text{for  $x,y\in\R.$} 
 \end{align*}
 Letting $H$ denote the Hilbert transform \cite{St93}, it holds that $B_{0,1}(0)=\pi H,$ and moreover
 \begin{align}\label{FOF}
 \pi\bA(f)=f'B_{0,1}(f)-B_{1,1}(f)[f,\,\cdot\,].
 \end{align}
 We first establish the following result.
 
 \begin{lemma}\label{L:99a}
 Let $1\leq m\in\N$ and  $n\in\N$ be given.  Then: 
 \begin{itemize}
  \item[$(i)$] Given Lipschitz functions $a_1,\ldots, a_{m},b_1,\ldots, b_n:\R\to\R,$ there exists  a positive constant  $C$,
  which depends only on $n,$ $m,$ and $\max_{i=1,\ldots, m}\|a_i'\|_\infty,$ such that 
\begin{align*}
 \|B_{n,m}(a_1,\ldots, a_{m})[b_1,\ldots, b_n,\oo]\|_2\leq C\|\oo\|_2\prod_{i=1}^{n} \|b_i'\|_\infty 
\end{align*}
for all $\oo\in L_2(\R).$
In particular $B_{n,m}(a_1,\ldots, a_{m})[b_1,\ldots, b_n,\,\cdot\,]\in \kL(  L_2(\R))$.

  \item[$(ii)$]  $B_{n,m}\in C^{1-}((W^1_\infty(\R))^{m},\kL_{n+1}((W^1_\infty(\R))^{n}\times L_2(\R),L_2(\R))).$   
   \item[$(iii)$]   Given $r\in (3/2,2)$ and $\tau\in(1/2,2),$  it holds
   \begin{align*}   
 \|B_{n,m}(a_1,\ldots, a_{m})[b_1,\ldots, b_n,\oo]\|_\infty\leq C\|\oo\|_{H^\tau} \prod_{i=1}^{n} \|b_i\|_{H^r} 
\end{align*}
  for all $a_1,\ldots, a_{m}, b_1,\ldots, b_n\in H^r(\R)$ and $\oo\in H^\tau(\R)$,  with $C$ depending only on $\tau,$ $ r,$ $n $, $m,$ and  $\max_{i=1,\ldots, m}\|a_i \|_{H^r}$.
  
 In particular,   $B_{n,m}\in C^{1-}((H^r(\R))^{m},\kL_{n+1}((H^r(\R))^{n}\times H^\tau(\R),L_\infty(\R))).$ 
 \end{itemize}  
 \end{lemma}
\begin{proof}
 The assertion $(i)$ has been proved in \cite[Remark~3.3]{M16x} by exploiting a result from harmonic analysis due to T. Murai \cite{TM86}. 
 Furthermore, the local Lipschitz continuity properties stated at $(ii)$ and $(iii)$ follow  from the  estimates at $(i)$ and $(iii)$, respectively, via the relation
 \begin{align}
 &\hspace{-1cm}B_{n,m}(\wt a_{1}, \ldots, \wt a_{m})[b_1,\ldots, b_{n},\oo]-  B_{n,m}(a_1, \ldots, a_{m})[b_1,\ldots, b_{n},\oo]\nonumber\\[1ex]
 &=\sum_{i=1}^{m} B_{n+2,m+1}(\wt a_{1},\ldots, \wt a_{i},a_i,\ldots \ldots, a_{m})[b_1,\ldots, b_{n},a_i+\wt a_{i}, a_i-\wt a_{i},\oo ].\label{rell}
\end{align}

In order to establish the estimate given  at $(iii)$, we write  
\begin{align*}
 B_{n,m}(a_1,\ldots, a_{m})[b_1,\ldots, b_n,\oo]= T_1+T_2+T_3,
\end{align*}
where
\begin{align*}
T_1 (x)&:=\int\limits_{|y|\leq 1} \cfrac{\prod_{i=1}^{n}\big(\delta_{[x,y]} b_i /y\big)}{\prod_{i=1}^{m}\big[1+\big(\delta_{[x,y]}  a_i /y\big)^2\big]} \frac{\oo(x-y)-\oo(x)}{y}\, dy,\\[1ex]
 T_2(x)&:= \oo(x)\PV\int\limits_{|y|\leq 1} \frac{1}{y}\cfrac{\prod_{i=1}^{n}\big(\delta_{[x,y]} b_i /y\big)}{\prod_{i=1}^{m}\big[1+\big(\delta_{[x,y]}  a_i /y\big)^2\big]}\, dy,\\[1ex]
 T_3(x)&:= \PV\int\limits_{|y|> 1}\cfrac{\prod_{i=1}^{n}\big(\delta_{[x,y]} b_i /y\big)}{\prod_{i=1}^{m}\big[1+\big(\delta_{[x,y]}  a_i /y\big)^2\big]} \frac{\oo(x-y)}{y}\, dy.
\end{align*}
A straightforward argument shows that 
\begin{align*}
 \|T_1\|_\infty\leq \frac{4}{2\tau-1}[\oo]_{\tau-1/2}\prod_{i=1}^{n}\|b_i'\|_\infty  \qquad \text{and}\qquad \|T_3\|_\infty\leq 2\|\oo\|_2\prod_{i=1}^{n}\|b_i'\|_\infty.
\end{align*}
Moreover, since $r-1/2\in (1,2)$,   it holds that $H^r(\R)\hookrightarrow {\rm BC}^{r-1/2}(\R) $,  and therefore
 \begin{align}\label{MM}
  \frac{|f(x+y)-2f(x)+f(x-y)|}{|y|^{r-1/2}}\leq 4[f']_{r-3/2} \qquad\text{for all $f\in H^r(\R)$, $x\in\R$, $y\neq0,$}
\end{align}
cf. \cite[Relation (0.2.2)]{L95}. Here $[\,\cdot\,]_{r-3/2}$ denotes the usual H\"older seminorm.  
Using the definition of the principal value together with \eqref{MM}, it follows that  
\begin{align*}
\|T_2\|_\infty\leq \frac{8}{r-3/2}\|\oo\|_\infty    \Big[\sum_{i=1}^{n} \Big([b_i']_{r-3/2}\prod_{{j=1, j\neq i}}^{n}\|b_j'\|_\infty\Big)+ \Big(\prod_{i=1}^{n}\|b_i'\|_\infty\Big)\sum_{i=1}^{m}\|a_i'\|_\infty[a_i']_{r-3/2}\Big],
\end{align*}
and  $(iii)$ follows at once.
\end{proof}

 We are additionally confronted in our analysis with a different type of singular integral operators. 
 These operators, denoted by $\ov B_{n,m}$, with $nm\geq1$,
 are  extensions of the operators $B_{n,m}$ introduced above
to a   Sobolev space product where a lower regularity of the  variable $b_1$ is compensated  by a higher regularity of the variable $\oo$.
The extension property is a consequence of the estimate \eqref{REF1} derived below, while the estimate  \eqref{REF2} plays a key role  later on 
in the proofs of the Theorems~\ref{TK1} and \ref{T2}, when identifying the important terms that need to be estimated.

\begin{lemma}\label{L:99d}
 Let  $ n,m\in\N$ with $nm\geq1,$    $\tau\in (1/2,1),$ and $r\in(5/2-\tau,2) $  be given.  
 \begin{itemize}
  \item[$(i)$] Given $a_1,\ldots, a_{m}\in H^r(\R)$,     we let
  \[
  \ov B_{n,m}(a_1,\ldots, a_{m})[b_1,\ldots, b_n,\oo]:=B_{n,m}(a_1,\ldots, a_{m})[b_1,\ldots, b_n,\oo]
  \]
  for $\oo\in H^1(\R),$ and $ b_1,\ldots, b_n\in H^r(\R).$
  Then, there exists a constant  $C$, depending only on $n,m$, $r$, $\tau$,  and $\max_{1\leq i\leq m}\|a_i\|_{H^r}$, such that
\begin{align} 
&\|\ov B_{n,m}(a_1,\ldots, a_{m})[b_1,\ldots, b_n,\oo]\|_2\leq C\|\oo\|_{H^\tau}\|b_1\|_{H^1}\prod_{i=2}^n\|b_i\|_{H^r} \label{REF1}
\end{align}
and
\begin{align} 
&\|\ov B_{n,m}(a_1,\ldots, a_{m})[b_1,\ldots, b_n,\oo]-B_{n-1,m}(a_1,\ldots, a_{m})[b_2,\ldots, b_n,b_1'\oo]\|_2\nonumber\\[1ex]
&\hspace{5.6cm}\leq C\|b_1\|_{H^{\tau}}\|\oo\|_{H^1}\prod_{i=2}^n\|b_i\|_{H^r}.\label{REF2}
\end{align}
In particular,  $\ov B_{n,m}(a_{1}, \ldots, a_{m})$ extends to a bounded operator   
$$\ov B_{n,m}(a_{1}, \ldots, a_{m}) \in\kL_{n+1}(H^1(\R)\times  (H^r(\R))^{n-1}\times H^{\tau}(\R), L_2(\R)).$$
  \item[$(ii)$]  $\ov B_{n,m}\in C^{1-}((H^r(\R))^m,\kL_{n+1}(H^1(\R)\times (H^{r}(\R))^{n-1}\times H^{\tau}(\R), L_2(\R))).$
 \end{itemize}  
 \end{lemma}
\begin{proof} Similarly as in the previous lemma, the assertion $(ii)$ is a consequence  of $(i)$, more precisely of \eqref{REF1}. 
To establish $(i)$ we use the formula
\[
\frac{\p}{\p y}\Big(\frac{\delta_{[x,y]}h}{y}\Big)=\frac{h'(x-y)}{y}-\frac{\delta_{[x,y]}h}{y^2}\qquad \text{for $x\in\R,$ $y\neq 0$,}
\]
and compute that 
\begin{align*}
  \ov B_{n,m}(a_1,\ldots, a_{m})[b_1,\ldots, b_n,\oo](x)&=\PV\int_\R\frac{\delta_{[x,y]} b_1}{y^2}\cfrac{\prod_{i=2}^{n}\big(\delta_{[x,y]} b_i /y\big)}{\prod_{i=1}^{m}\big[1+\big(\delta_{[x,y]}  a_i /y\big)^2\big]} \oo(x-y)\, dy\\[1ex]
  &=B_{n-1,m}(a_1,\ldots, a_{m})[b_2,\ldots, b_n,b_1'\oo](x)\\[1ex]
  &\hspace{0.424cm}-\PV\int_\R\frac{\p}{\p y}\Big(\frac{\delta_{[x,y]}b_1}{y}\Big)\cfrac{\prod_{i=2}^{n}\big(\delta_{[x,y]} b_i /y\big)}{\prod_{i=1}^{m}\big[1+\big(\delta_{[x,y]}  a_i /y\big)^2\big]} \oo(x-y)\, dy.
\end{align*}
Integrating  by parts, we are led  to the relation
\begin{align}
  &\hspace{-1cm} \PV\int_\R\frac{\p}{\p y}\Big(\frac{\delta_{[x,y]}b_1}{y}\Big)\cfrac{\prod_{i=2}^{n}\big(\delta_{[x,y]} b_i /y\big)}{\prod_{i=1}^{m}\big[1+\big(\delta_{[x,y]}  a_i /y\big)^2\big]} \oo(x-y)\, dy\nonumber\\[1ex]
  &=(b_1B_{n-1,m}(a_1,\ldots, a_{m})[b_2,\ldots, b_n, \oo'])(x)-B_{n-1,m}(a_1,\ldots, a_{m})[b_2,\ldots, b_n, b_1\oo'](x)\nonumber\\[1ex]
  &\hspace{0.424cm}+\sum_{j=2}^n\int_\R  K_{1,j}(x,y)  \oo(x-y)\, dy-\sum_{j=1}^m\int_\R K_{2,j}(x,y) \   \oo(x-y)\, dy,\label{BGG}
\end{align}
where, for $x\in\R$ and $y\neq0$, we set
\begin{align*}
 K_{1,j}(x,y)&:= \cfrac{\prod_{i=2, i\neq j }^{n}\big(\delta_{[x,y]} b_i /y\big)}{\prod_{i=1}^{m}\big[1+\big(\delta_{[x,y]}  a_i /y\big)^2\big]}\frac{\delta_{[x,y]}b_j-yb_j'(x-y)}{y^2}\frac{\delta_{[x,y]}b_1}{y},\\[1ex]
 K_{2,j}(x,y)&:= 2\cfrac{\prod_{i=2 }^{n}\big(\delta_{[x,y]} b_i /y\big)}{\big[1+\big(\delta_{[x,y]}  a_j /y\big)^2\big]\prod_{i=1}^{m}\big[1+\big(\delta_{[x,y]}  a_i /y\big)^2\big]}
  \frac{\delta_{[x,y]}a_j-ya_j'(x-y)}{y^2}\frac{\delta_{[x,y]}a_j}{y}\frac{\delta_{[x,y]}b_1}{y}.
\end{align*}
In view of Lemma~\ref{L:99a} we have 
\begin{align}\label{BGG0}
 \|B_{n-1,m}(a_1,\ldots, a_{m})[b_2,\ldots, b_n,b_1'\oo]\|_2\leq C\|\oo\|_\infty\|b_1\|_{H^1}\prod_{j=2}^n\|b_j'\|_\infty,
\end{align}
and we are left to estimate the $L_2$-norm of the terms on  the right-hand side of \eqref{BGG}.
For the integral terms we use
Minkowski's  integral inequality to obtain that
\begin{align*}
\Big(\int_\R\Big|\int_\R K_{1,j}(x,y)\oo(x-y)\, dy\Big|^2\, dx\Big)^{1/2}\leq  \|\oo\|_{\infty}\int_\R\Big(\int_\R |K_{1,j}(x,y)|^2\, dx\Big)^{1/2}\, dy.
\end{align*}
In the following $\kF$ denotes the Fourier transform.
Appealing to  $b_1\in{\rm BC}^{\tau-1/2}(\R)$, we  have 
\begin{align*}
 \int_\R |K_{1,j}(x,y)|^2\, dx =&\frac{1 }{y^{7-2\tau}} [b_1]_{\tau-1/2}^2\Big(\prod_{i=2, i\neq j}^{n}\|b_i'\|_{\infty}^2\Big)\int_\R|b_j-\tau_y b_j-y\tau_y b_j'|^2\, dx\\[1ex]
 \leq& \frac{C }{y^{7-2\tau}} \|b_1\|_{H^{\tau}}^2\Big(\prod_{i=2, i\neq j}^{n}\|b_i'\|_{\infty}^2\Big)\int_\R|\mathcal{F} b_j(\xi)|^2|e^{iy\xi}-1-iy\xi|^2\, d\xi,
\end{align*}
and together with the  inequality 
\begin{align*}
|e^{iy\xi}-1-iy\xi|^2\leq  C\big[(1+|\xi|^2)^ry^{2r} {\bf 1}_{(-1,1)}(y)+y^2(1+|\xi|^2) {\bf 1}_{[|y|\geq1]}(y)\big],\qquad y,\xi\in\R,
\end{align*}
we find   
\[
\int_\R |K_{1,j}(x,y)|^2\, dx\leq C \|b_1\|_{H^{\tau}}^2 \Big(\prod_{i=2}^{n}\|b_i\|_{H^r}^2\Big) \Big[y^{2(r+\tau)-7} {\bf 1}_{(-1,1)}(y)+ \frac{1}{y^{5-2\tau}} {\bf 1}_{[|y|\geq1]}(y)\Big],\qquad 2\leq j\leq n.
\]
Consequently, 
\begin{align}\label{BGG2}
\Big(\int_\R\Big|\int_\R K_{1,j}(x,y)\oo(x-y)\, dy\Big|^2\, dx\Big)^{1/2}\leq C\|\oo\|_\infty \|b_1\|_{H^{\tau}} \Big(\prod_{i=2}^{n}\|b_i\|_{H^r}\Big),\qquad 2\leq j\leq n,
\end{align} 
and analogously we obtain for $ 1\leq j\leq m$ that
\begin{align}\label{BGG3}
\Big(\int_\R\Big|\int_\R K_{2,j}(x,y)\oo(x-y)\, dy\Big|^2\, dx\Big)^{1/2}\leq   C\|\oo\|_\infty\|a_j\|_{H^r} \|b_1\|_{H^{\tau}} \Big(\prod_{i=2}^{n}\|b_i\|_{H^r}\Big).
\end{align}

We are left  with the term
\[T:=b_1B_{n-1,m}(a_1,\ldots, a_{m})[b_2,\ldots, b_n, \oo'] -B_{n-1,m}(a_1,\ldots, a_{m})[b_2,\ldots, b_n, b_1\oo'].\]
The estimate  \eqref{REF2} follows  by using Lemma~\ref{L:99a}. In order to derive \eqref{REF1} we proceed differently. 
The relation $\oo'(x-y)=(\p/\p y) (\oo(x)-\oo(x-y))$ together with integration by parts leads us to  
\begin{align*}
T(x)=\sum_{j=1}^n \int_\R K_{3,j}(x,y)\, dy -2\sum_{j=1}^m\int_\R K_{4,j}(x,y)\, dy,
\end{align*}
where
\begin{align*}
K_{3,j}(x,y)&:= \frac{\prod_{i=1, i\neq j}^n \delta_{[x,y]}b_i/y}{\prod_{i=1}^m 1+ \big(\delta_{[x,y]}a_i/y\big)^2}\frac{\delta_{[x,y]}\oo}{y}\Big(\frac{\delta_{[x,y]}b_j}{y}- b_j'(x-y)\Big),\\[1ex]
K_{4,j}(x,y)&:= \frac{\prod_{i=1}^n \delta_{[x,y]}b_i/y}{\big[1+ \big(\delta_{[x,y]}a_j/y\big)^2\big]\prod_{i=1}^m 1+ \big(\delta_{[x,y]}a_i/y\big)^2}\frac{\delta_{[x,y]}\oo}{y}\frac{\delta_{[x,y]}a_j}{y}\Big(\frac{\delta_{[x,y]}a_j}{y}- a_j'(x-y)\Big) .
\end{align*}

We first estimate the integrals defined by  the kernels $K_{4,j},$ $1\leq j\leq m.$
Minkowski's  integral inequality  implies that 
\begin{align*}
\Big(\int_\R\Big|\int_\R K_{4,j}(x,y) \, dy\Big|^2\, dx\Big)^{1/2}\leq    \int_\R\Big(\int_\R |K_{4,j}(x,y)|^2\, dx\Big)^{1/2}\, dy,
\end{align*}
and, since  $b_1\in{\rm BC}^{r-3/2}(\R)$, we get
\begin{align*}
 \int_\R |K_{4,j}(x,y)|^2\, dx&\leq \frac{2 }{y^{7-2r}} [b_1]_{{r-3/2}}^2\|a_j'\|_\infty^2\Big(\prod_{i=2}^{n}\|b_i'\|_{\infty}^2\Big)\int_\R|\oo-\tau_y\oo|^2\, dx\\[1ex]
 &= \frac{C }{y^{7-2r}} \|b_1\|_{H^{r-1}}^2\|a_j'\|_\infty^2\Big(\prod_{i=2}^{n}\|b_i'\|_{\infty}^2\Big)\int_\R|\mathcal{F}\oo(\xi)|^2|e^{iy\xi}-1|^2\, d\xi.
\end{align*}
Tanking advantage of  the inequality
\[
|e^{iy\xi}-1|^2\leq  C\big[(1+|\xi|^2)^{\tau}y^{2\tau} {\bf 1}_{(-1,1)}(y)+ {\bf 1}_{[|y|\geq1]}(y)\big],\qquad y,\xi\in\R,
\]
it follows that
\[
\int_\R |K_{4,j}(x,y)|^2\, dx\leq C \|b_1\|_{H^{r-1}}^2\|\oo\|_{H^{\tau}}^2\|a_j'\|_\infty^2\Big(\prod_{i=2}^{n}\|b_i'\|_{\infty}^2\Big)  \Big[y^{2(r+\tau)-7} {\bf 1}_{(-1,1)}(y)+ \frac{1}{y^{7-2r}}{\bf 1}_{[|y|\geq1]}(y)\Big],
\]
and therewith 
\begin{align}\label{BGG4}
 \Big(\int_\R\Big|\int_\R K_{4,j}(x,y) \, dy\Big|^2\, dx\Big)^{1/2}\leq C\|   w\|_{H^\tau}\|b_1\|_{H^{r-1}}\|a_j'\|_\infty\prod_{i=2}^n\|b_i'\|_\infty,\qquad 1\leq j\leq m.
\end{align}
Analogously, for $2\leq j\leq n,$ we have that
\begin{align}\label{BGG5}
 \Big(\int_\R\Big|\int_\R K_{3,j}(x,y) \, dy\Big|^2\, dx\Big)^{1/2}\leq C \|w\|_{H^\tau}\|b_1\|_{H^{r-1}}\prod_{i=2}^n\|b_i'\|_\infty, 
\end{align}
while, for $j=1$, we obtain the following estimate
\begin{align}\label{BGG6}
\Big(\int_\R\Big|\int_\R K_{3,1}(x,y) \, dy\Big|^2\, dx\Big)^{1/2}&\leq C \|w\|_{H^\tau}\|b_1\|_{H^{r-1}}\prod_{i=2}^n\|b_i'\|_\infty\nonumber\\[1ex]
&\hspace{0.424cm}+\|w\|_\infty \|B_{n-1,m}(a_1,\ldots,a_m)[b_2,\ldots, b_n, b_1']\|_2\nonumber\\[1ex]
&\hspace{0.424cm}+\|B_{n-1,m}(a_1,\ldots,a_m)[b_2,\ldots, b_n, b_1'w]\|_2.
\end{align}
The  estimate \eqref{REF1} follows from Lemma~\ref{L:99a} and \eqref{BGG0}-\eqref{BGG6}. 
\end{proof}\medskip

\paragraph{\bf Mapping properties}
Using the  Lemmas \ref{L:99a}-\ref{L:99d}, we now study the mapping properties of the nonlinear (with respect to $f$) operator $\bA$ defined in \eqref{DL1}.

\begin{lemma}\label{L:99e} 
It holds   that
\begin{align} \label{REG0}
 \bA\in C^{1-}(H^2(\R),\kL(H^1(\R)).
\end{align}
\end{lemma}
\begin{proof}
Given $f\in H^2(\R)$,  the relation \eqref{FOF} together with Lemma~\ref{L:99a} implies that $\bA(f)\in\kL(L_2(\R)).$
We next show that $\bA(f)[\oo]\in H^1(\R),$ provided that $\oo\in H^1(\R)$.
To this end we let $\{\tau_\e\}_{\e\in\R}$ denote the  $C_0$-group  of right  translations on $L_2(\R)$, that is $\tau_\e f(x):=f(x-\e) $ for $f\in L_2(\R)$, $x,\e\in\R,$ and  we compute for $\e\in(0,1)$ that
\begin{align*}
 \pi\frac{\tau_\e(\bA(f)[\oo]) -\bA(f)[\oo] }{\e}&=\frac{\tau_\e f' -f'}{\e}B_{0,1}(\tau_\e f)[\tau_\e\oo]+f 'B_{0,1}(\tau_\e f )\Big[\frac{\tau_\e \oo -\oo}{\e}\Big]\\[1ex]
 &\hspace{0.424cm}-f'\ov B_{2,2}(\tau_\e f,f)\Big[ \frac{\tau_\e f -f}{\e},\tau_\e f+f,\oo\Big]-\ov B_{1,1}(\tau_\e f)\Big[\frac{\tau_\e f -f}{\e}, \tau_\e\oo\Big]\\[1ex]
 &\hspace{0.424cm}-B_{1,1}(\tau_\e f )\Big[f,\frac{\tau_\e \oo -\oo}{\e}\Big]- \ov B_{3,2}(\tau_\e f,f)\Big[ \frac{\tau_\e f -f}{\e},\tau_\e f+f,f,\oo\Big].
\end{align*}
The convergences 
\begin{align*}
&\tau_\e f\underset{\e\to0}\longrightarrow f  \quad\text{in $H^2(\R),$ } \qquad   \frac{\tau_\e f - f }{\e}\underset{\e\to0}{\longrightarrow} -f' \quad\text{in $H^{1}(\R),$ }\\[1ex]
&\tau_\e \oo\underset{\e\to0}\longrightarrow \oo  \quad\text{in $H^1(\R),$ } \qquad   \frac{\tau_\e \oo - \oo }{\e}\underset{\e\to0}{\longrightarrow} -\oo' \quad\text{in $L_2(\R),$ }
\end{align*}
together with the Lemmas~\ref{L:99a}-\ref{L:99d} enable us to pass to the limit $\e\to0$ in the above relation and to conclude that $\bA(f)[\oo]\in H^1(\R)$ and
\begin{align}
 \pi(\bA(f)[\oo])'&=\pi\bA(f)[\oo']+f''B_{0,1}( f)[\oo]-2f'\ov B_{2,2}(f,f)[f',f,\oo]-\ov B_{1,1}( f)[f', \oo]\nonumber\\[1ex]
 &\hspace{0.424cm}- 2\ov B_{3,2}( f,f)\Big[f',f,f,\oo\Big].\label{derA}
\end{align}
The Lipschitz continuity property \eqref{REG0} is now a direct consequence of  the Lemmas~\ref{L:99a}-\ref{L:99d}.
\end{proof}

In fact,  $\bA$ enjoys the following regularity
\begin{align}\label{RegA}
 \bA\in C^\omega(H^r(\R),\kL(L_2(\R)))\cap C^{\omega} (H^{2}(\R),\kL(H^1(\R))) \qquad\text{for all $r>3/2$.}
\end{align}
The property \eqref{RegA} may be established by using the arguments presented in \cite[Section 5]{M16x} together with the  Lemmas~\ref{L:99a}-\ref{L:99d}.
The lengthy details are left to the interested reader.\medskip

Given $f\in H^2(\R)$, we denote by $\bB(f)$ the operator which corresponds to the right-hand side of the first equation of \eqref{P:1}, namely 
\begin{align}\label{oper}
 \bB(f):=B_{0,1}(f)+f'B_{1,1}(f)[f,\,\cdot\,].
\end{align}
For later purposes we establish the following regularity result.

\begin{lemma}\label{L:reg} It holds that 
\begin{align}\label{regB}
\bB\in C^{\omega}(H^r(\R), \kL(L_2(\R)))\cap C^{\omega}(H^2(\R), \kL(H^1(\R)))\qquad\text{for all $r>3/2$.}
\end{align}
\end{lemma}
\begin{proof} 
Given $f\in H^r(\R)$, with $r>3/2,$ it follows from Lemma~\ref{L:99a} that $\bB(f)\in\kL(L_2(\R)).$
Furthermore, proceeding as in Lemma~\ref{L:99e}, we get, in view of the Lemmas~\ref{L:99a}-\ref{L:99d},  that if $f\in H^2(\R)$ and $\oo\in H^1(\R)$, then $\bB(f)[\oo]\in H^1(\R)$ and
 \begin{align} 
  (\bB(f)[\oo])'&=\bB(f)[\oo']-2\ov B_{2,2摤}(f,f)[f',f,\oo]+f''B_{1,1}(f)[f,\oo]+f'\ov B_{1,1}(f)[f',\oo]\nonumber\\[1ex]
  &\hspace{0.424cm}-2f'\ov B_{3,2}(f,f)[f',f,f,\oo].\label{derB}
 \end{align}
The analyticity  property \eqref{regB} follows  now from the  Lemmas~\ref{L:99a}-\ref{L:99d} by arguing as in \cite[Section~5]{M16x} (we omit again the lengthy details).
\end{proof}\medskip

 \paragraph{\bf On the resolvent set of the adjoint of the double layer potential}
We  are now in the position to address the solvability of \eqref{P:2}.   
First, we  show that, given $f\in H^r(\R)$, with $r>3/2$,    the resolvent set of $ \bA(f)$,  when we view $ \bA(f)$ as an operator in $\kL(L_2(\R)),$ contains the set $[|\lambda|\geq1]\cap\R$. 
This property follows from the uniform bound  (with respect to $f$ and $\lambda$) establish in \eqref{DI}. 
This bound proves crucial also when investigating the resolvent set of $\bA(f)$,  when we regard  $ \bA(f)$ as an element of $\kL(H^1(\R)) $ and $f\in H^2(\R)$, cf. Proposition~\ref{P:I2} below.
\begin{thm}\label{T:I1} 
Given  $f\in H^r(\R)$, $r>3/2$,  and $\lambda\in\R$ with $|\lambda|\geq1$, it holds  
\begin{align*} 
 \lambda-\bA(f)\in{\rm Isom\,}(L_2(\R)).
\end{align*}
\end{thm}
\begin{proof}
Let  $M>0$ be given. We prove that there exists a constant $C=C(M)>0$ such that for all   $f\in H^r(\R)$ with $\|f\|_{H^r}\leq M$,  $\ov\omega\in L_2(\R)$, and $\lambda\in\R$ with $|\lambda|\geq1$ we have
 \begin{align}\label{DI}
\|\ov\omega\|_{2}\leq C\|(\lambda-\bA(f))[\ov\omega]\|_2.
\end{align}
Having established \eqref{DI}, the claim follows from  the method of continuity, cf. \cite[Proposition 1.1.1]{Am95}, as, for each  $f\in H^r(\R)$,  the spectrum of $\bA(f)$ is compact and therefore
$\lambda-\bA(f)$ is  invertible  if $\lambda$ is large. 
In view of \eqref{RegA}, it suffices to prove \eqref{DI} for   $f\in C^\infty_0(\R)$   and   $\ov\omega\in C_0^\infty(\R)$.

We recall from the Lemma \ref{L:A2} that the restrictions $\wh v_\pm:=\wh v\big|_{\0_\pm^0}$, with $\0_\pm^0$ as in Lemma \ref{L:A2}, of the function $\wh v$ defined in \eqref{V1'}
have the following properties
\begin{align}\label{diro}
 \text{$ \wh v_\pm  \in  C (\ov{\0_\pm^0})\cap C^1({\0_\pm^0})$ \qquad and\qquad  ${\rm div \, } \wh v_\pm={\rm rot \, } \wh v_\pm=0$\quad\text{in $\0_\pm^0$.}}
\end{align}
Moreover,    the Plemelj formula \eqref{V2}  together with the Lemmas~\ref{L:99e}-\ref{L:reg}  ensures the restrictions $F_\pm:=(F_\pm^{1}, F_\pm^2):=\wh v_\pm\big|_{[y=f(x)]}$  satisfy
\begin{align}\label{BB1}
  \langle F_\pm|(1,f')\rangle=\frac{1}{2}(\bA(f)\mp 1 )[\oo]\in H^1(\R),\qquad {\rm F}^\nu:=\langle F_\pm|(-f',1)\rangle=\frac{1}{2\pi}\bB(f)[\oo]\in H^1(\R).
\end{align}
Letting $\tau$ and $\nu:=(\nu_1,\nu_2)$ denote the tangent and the   unit outward normal vectors at $\p \0_-^0$, we write
$F_\pm=F_\pm^\tau+F_\pm^\nu$, where 
\[\text{$F_\pm^\tau:=\langle F_\pm|\tau\rangle\tau$\qquad  and \qquad  $F_\pm^\nu:=\langle F_\pm|\nu\rangle\nu=\frac{{\rm F}^\nu}{\sqrt{1+f'^2}}\nu\in H^1(\R,\R^2)$.}\]
We now introduce the bilinear form $\B:L_2(\R,\R^2)\times L_2(\R,\R^2)\to\R$ by the formula
\[
\B(F,G):=\int_{\R}G^2\langle F|(-f',1) \rangle +F^2\langle G|(-f',1) \rangle -\langle  F|G \rangle\, dx,\qquad F=(F^1,F^2),\, G=(G^1,G^2),
\]
 and we remark that
\begin{align*}
 (i)&\quad \B(F,G)=\int_{\R} \langle F|  G  \rangle   \, dx \qquad\text{if $F=\langle F|\nu\rangle\nu$, $G=\langle G|\nu\rangle\nu$,}\\[1ex]
 (ii)& \quad \B(F,G)=-\int_{\R}\langle F|\tau\rangle  \langle G|\tau \rangle   \, dx  \qquad\text{if $F=\langle F |\tau\rangle\tau$, $G=\langle G|\tau\rangle\tau$,}\\[1ex]
 (iii)& \quad \B(F,G)=\int_{\R} F^2 \langle G|(-f',1)\rangle  \, dx  \qquad\text{if $F=\langle F|\tau\rangle\tau$, $G=\langle G|\nu\rangle\nu$.}
\end{align*}
We now claim that 
\begin{align}\label{BB}
 \B(F_\pm,F_\pm)=0.
\end{align}
In order to prove \eqref{BB} we choose  a sequence  $(\varphi_n)_n\subset C^\infty_0(\R^2, [0,1]) $ with the property that $\varphi_n=1$ in $[|(x,y)|\leq n]$,  $\varphi_n=0$ in $[|(x,y)|\geq n+1]$,
and $\sup_{n}\|\nabla \varphi_n\|_\infty<\infty.$ Using Lebesgue's dominated convergence theorem and Stokes' theorem together with \eqref{diro}, we get
\begin{align*}
 \B(F_-,F_-)&=\int_{[y=f(x)]}\Big\langle\Big(\begin{array}{ccc}
                             2F_-^1F_-^2\\[0.1ex]
                             (F_-^2)^2-(F_-^1)^2
                            \end{array}\Big)\Big|\nu\Big\rangle\, d\sigma
                            =\lim_{n\to\infty}\int_{[y=f(x)]}\Big\langle\varphi_n\Big(\begin{array}{ccc}
                             2\wh v_-^1\wh v_-^2\\[0.1ex]
                             (\wh v_-^2)^2-(\wh v_-^1)^2
                            \end{array}\Big)\Big|\nu\Big\rangle\, d\sigma\\[1ex]                            
                            &=\lim_{n\to\infty}\int_{\0_-^0}{\rm div\, }\Big(\varphi_n \Big(\begin{array}{ccc}
                             2\wh v_-^1\wh v_-^2\\[0.1ex]
                             (\wh v_-^2)^2-(\wh v_-^1)^2
                            \end{array}\Big)\Big) \, d(x,y)\\[1ex]
                            &= \lim_{n\to\infty}\int_{\0_-^0}\Big\langle\nabla \varphi_n\Big| \Big(\begin{array}{ccc}
                             2\wh v_-^1\wh v_-^2\\[0.1ex]
                             (\wh v_-^2)^2-(\wh v_-^1)^2
                            \end{array}\Big) \Big\rangle\, d(x,y)=0,
\end{align*}
the last equality being a consequence of  \eqref{LA3} and of  $\sup_{n}\|\nabla \varphi_n\|_\infty<\infty$. 
This proves \eqref{BB} for $F_-$. The proof for $F_+$ is similar.

Using \eqref{BB1}, \eqref{BB}, and the relations $(i)-(iii)$, we now obtain the following Rellich formula 
\begin{align}\label{Ea1}
 0&=\B(F_\pm,F_\pm)= \B(F_\pm^\tau,F_\pm^\tau)+2 \B(F_\pm^\tau,F_\pm^\nu)+ \B(F_\pm^\nu,F_\pm^\nu)\nonumber\\[1ex]
 &=\int_\R\frac{1}{1+f'^2}\Big[| {\rm F}^\nu|^2+f' {\rm F}^\nu(\bA(f)\mp1)[\oo]-    \frac{1}{4}|(\bA(f)\mp1)[\oo]|^2\Big]\, dx.
 \end{align}
    Young's inequality and \eqref{Ea1} imply that there exists a positive constant $C=C(M)$ such that 
\begin{align*} 
\|(\bA(f)\pm1)[\oo]\|_2\leq C\| {\rm F}^\nu\|_2
\end{align*}
for all $f\in C^\infty_0(\R)$ with $\|f\|_{H^r}\leq M$ and   $\ov\omega\in C_0^\infty(\R).$
The latter inequality yields in particular 
\begin{align} \label{pm2}
  \|\oo\|_2=\frac{1}{2}\|(\bA(f)+1)[\oo]-(\bA(f)-1)[\oo]\|_2\leq C\| {\rm F}^\nu\|_2.
\end{align}

Let  $\lambda\in\R$ satisfy $|\lambda|\geq1$. Since
\begin{align*}
  |(\bA(f)\pm1)[\oo]|^2=  |(\lambda-\bA(f))[\oo] |^2-2(\lambda\pm1)\oo(\lambda-\bA(f))[\oo] + (\lambda\pm 1)^2|\oo|^2,
\end{align*}
  we deduce, together with \eqref{Ea1}, after  eliminating the mixed term, that  
\begin{align*} 
 \int_\R\frac{1}{1+f'^2}\big[(\lambda^2-1)|\oo|^2+4| {\rm F}^\nu|^2\big]\, dx=\int_\R\frac{1}{1+f'^2}\big[|(\lambda-\bA(f))[\oo]|^2+4f' {\rm F}^\nu(\lambda-\bA(f))[\oo] \big]\, dx.
 \end{align*}
 This relation together with Young's inequality allows us to conclude that there exists a constant  $C=C(M)$ with the property that
 \[
 (\lambda^2-1)\|\oo\|_2+\| {\rm F}^\nu\|_2\leq C \|(\lambda-\bA(f))[\oo]\|_2.
 \]
 The desired property \eqref{DI} follows from \eqref{pm2}.
 \end{proof}

We are now in the position to study the  invertibility of $\lambda-\bA(f)$ in   $\kL(H^1(\R))$, when  requiring additionally that $f\in H^2(\R).$
 
\begin{prop}\label{P:I2} 
Given  $f\in H^2(\R)$  and $\lambda\in\R$ with $|\lambda|\geq 1,$  it holds   that
\begin{align*} 
 \lambda-\bA(f)\in {\rm Isom}(H^1(\R)).
\end{align*}
\end{prop}
\begin{proof}
As in the previous theorem, it suffices to prove that, given $M>0$, there exists a constant $C_1=C_1(M)$ such that for all $f\in H^2(\R)$ with $\|f\|_{H^2}\leq M$, $\lambda\in\R$ with $|\lambda|\geq1$, and $\oo\in H^1(\R)$ we have  
 \begin{align}\label{FE}
  \|\oo\|_{H^1}\leq C_1\|(\lambda-\bA(f))[\oo]\|_{H^1}.
 \end{align}
 In view of \eqref{DI}, we are left to estimate  the term $\|((\lambda-\bA(f))[\oo])'\|_2$. 
 To this end, we  infer from \eqref{derA} that 
 \begin{align}\label{DERR}
       ((\lambda-\bA(f))[\oo])'=(\lambda-\bA(f))[\oo'] -T_{0,\rm lot}(f)[\oo], 
 \end{align}
 where, given $f\in H^2(\R)$ and $\oo \in H^1(\R),$ we have set
 \begin{align}
  \pi T_{0,\rm lot}(f)[\oo]&:=f''B_{0,1}( f)[\oo]-2f'\ov B_{2,2}(f,f)[f',f,\oo]-\ov B_{1,1}( f)[f', \oo]\nonumber\\[1ex]
  &\hspace{0.424cm}- 2\ov B_{3,2}( f,f)[f',f,f,\oo].\label{0lot}
 \end{align}
We now fix $\tau\in (1/2,1)$ and  $r\in (5/2-\tau,2)$. 
The Lemmas~\ref{L:99a}-\ref{L:99d} ensure that there exists a constant $C_0=C_0(M)$ such that 
  \begin{align}\label{0lote}
 \|T_{0,\rm lot}(f)[\oo] \|_2\leq C_0 \|\oo\|_{H^\tau} 
\end{align}
for all $f\in H^2(\R)$ with $\|f\|_{H^2}\leq M$ and all $\oo\in H^1(\R)$.
Using Young's inequality,   \eqref{DI}, \eqref{DERR},  \eqref{0lote}, and the inequality $\|\oo\|_{H^\tau}\leq \|\oo\|_2^{1-\tau}\|\oo\|_{H^1}^\tau,$ we conclude that 
\begin{align*}
 \|(\lambda-\bA(f))[\oo]\|_{H^1}\geq& \frac{1}{2C}\|\oo\|_{H^1}-C_0\|\oo\|_{H^\tau}\geq \frac{1}{4C}\|\oo\|_{H^1}-\wt C_0\|\oo\|_{2},
\end{align*}
where $C=C(M)$ denotes  the constant in \eqref{DI} and $\wt C_0$ depends only on $  C $ and $C_0$.
The estimate \eqref{FE} follows now by appealing once more to \eqref{DI}. 
\end{proof}

Arguing as in Proposition~\ref{P:I2}, it can be shown  that the following general result holds.
\begin{rem}\label{R:IR} 
Given  $f\in H^k(\R)$, $k\geq 2$,  and $\lambda\in\R$ with $|\lambda|\geq 1,$  it holds   that
\begin{align*} 
 \lambda-\bA(f)\in {\rm Isom}(H^{k-1}(\R)).
\end{align*}
\end{rem}
 
 We now establish an invertibility  result for the double layer potential, that is for the adjoint of $\bA(f)$, which is needed in the proof of Theorem \ref{MT2}. 
 
 \begin{lemma}\label{L:99f}
 Given  $f\in H^2(\R)$, the adjoint $(\bA(f))^*\in\kL(L_2(\R))$ of  $\bA(f)\in\kL(L_2(\R))$ is the operator
 \[
(\bA(f))^*[\varphi](x)=\frac{1}{\pi}\PV\int_\R\frac{(f(x)-f(x-y))-yf'(x-y)}{y^2+(f(x)-f(x-y))^2}\varphi(x-y)\, dy, \qquad \varphi\in L_2(\R).
\]
Moreover, for each $M>0$, there exists a constant $C=C(M)$ such that 
 \begin{align}\label{FE2}
  \|\varphi\|_{H^1}\leq C_1\|(\lambda-(\bA(f))^*)[\varphi]\|_{H^1}
 \end{align}
 for all $\|f\|_{H^2}\leq M$, $\lambda\in\R$ with $|\lambda|\geq 1,$  and $\varphi\in H^1(\R)$.
 In particular $\lambda-(\bA(f))^*\in {\rm Isom}(H^1(\R))$ for all $\lambda\in\R$ with $|\lambda|\geq 1.$  
\end{lemma}
\begin{proof}
 It is easy to verify that $  (\bA(f))^*\in\kL(L_2(\R))$ is indeed the adjoint of $\bA(f)\in\kL(L_2(\R))$, while the property that $\lambda-\bA(f)\in {\rm Isom}(L_2(\R))$ for all $\lambda\in\R$ with $|\lambda|\geq 1$ follows from Theorem \ref{T:I1}.
  Moreover, the Lemmas \ref{L:99a}-\ref{L:99d} imply that   $\lambda- (\bA(f))^*\in\kL(H^1(\R))$, with
 \[
 ((\lambda- (\bA(f))^*)[\varphi])'=((\lambda- (\bA(f))^*)[\varphi']-T_{0,\rm lot}^*(f)[\varphi],
 \]
 where
 \begin{align*}
 \pi T_{0,\rm lot}^*(f)[\varphi]:=-B_{0,1}(f)[f''\varphi]+\ov B_{1,1}(f)[f',\varphi]+ 2\ov B_{2,2}(f,f)[f',f,f'\varphi]- 2\ov B_{3,2}(f,f)[f',f,f,\varphi].
 \end{align*}
For $\|f\|_{H^2}\leq M,$  the Lemmas \ref{L:99a}-\ref{L:99d} lead us to the following estimate
\begin{align*}
 \|\pi T_{0,\rm lot}^*(f)[\varphi]\|_2\leq C_0\|\varphi\|_{H^{3/4}},
 \end{align*}
 with a constant $C_0=C_0(M)$. 
 Letting $C=C(M)$ denote the constant in  \eqref{DI}, we additionally get  
 \begin{align*}
\|\varphi\|_{2}\leq C\|(\lambda-(\bA(f))^*)[\varphi]\|_2 
\end{align*}
 for all $f\in H^2(\R)$ with $\|f\|_{H^2}\leq M$,  $\varphi\in L_2(\R)$, and $\lambda\in\R$ with $|\lambda|\geq1$,   and therefore we may argue as in the proof of Proposition \ref{P:I2}
 in order to obtain the remaining claims.
\end{proof}

\section{The Muskat problem with surface tension}\label{Sec3}
In this section we address the well-posedness of the Muskat problem  with surface tension, and therefore we assume throughout   that $\sigma>0.$ 
Taking advantage of Theorem~\ref{T:I1} and of the structure of the curvature term, we first formulate the system \eqref{P}  as a quasilinear evolution problem for the free boundary $f$ only, cf. \eqref{Pest1}.
Subsequently, we disclose the parabolic character of \eqref{Pest1}, and this enables us to use  abstract results for  quasilinear parabolic problems due to H. Amann \cite{Am93, Am86, Am88}, see \cite[Theorem 1.5]{M16x} for the 
precise statements. \medskip

\paragraph{\bf The abstract formulation}
We start by solving \eqref{P:2}.
For our approach it is important to  point  out the quasilinear structure of the equation \eqref{P:2}  which is a result  of the linearity of left-hand side of \eqref{P:2} with respect to the highest  derivative of $f$.
Concerning  this issue, it is convenient to study the solvability for $\oo$ of the equation  
\begin{equation}\label{P:2'}
b_\mu\Big[\sigma \frac{h'''}{(1+f'^2)^{3/2}}-3\sigma\frac{f'{f''}h''}{(1+f'^2)^{5/2}}-\Theta h'\Big]=(1+a_\mu\bA(f))[\ov\omega],
\end{equation}
where
\[
b_\mu:=\frac{2k}{\mu_-+\mu_+}.
\]
 If we  replace $h$ by $f$, then \eqref{P:2'} is clearly equivalent to \eqref{P:2}.
 Since  the values of the positive constants $b_\mu$ and $\sigma$ are not relevant for the further analysis we set $b_\mu=\sigma=1$.  

\begin{prop}\label{P1} Given $f\in H^2(\R) $   and $h\in H^3(\R)$,  there exists a unique solution $\oo:=\oo(f)[h]$ to  \eqref{P:2'} and
\begin{align}\label{RegO}
 \oo\in C^\omega(H^2(\R), \kL(H^3(\R), L_2(\R))).
\end{align}
\end{prop}
\begin{proof} Since $|a_\mu|<1$, it follows from Theorem~\ref{T:I1} that 
\begin{align*}
 \oo(f)[h]:= (1+a_\mu\bA(f))^{-1}\Big[\frac{h'''}{(1+f'^2)^{3/2}}-3\frac{f'{f''}h''}{(1+f'^2)^{5/2}}-\Theta h'\Big]\in L_2(\R)
\end{align*}
is the unique solution to \eqref{P:2'}.
 Since
 \begin{align*}
  &[T\mapsto T^{-1}]\in C^\omega\big({\rm Isom\,}(L_2(\R)), {\rm Isom\,}(L_2(\R))\big),\\[1ex]
  &\left[f\mapsto \Big[h\mapsto\frac{h'''}{(1+f'^2)^{3/2}}-3\frac{f'f''h''}{(1+f'^2)^{5/2}}-\Theta h'\Big]\right]\in C^\omega(H^2(\R),\kL(H^3(\R),L_2(\R))),
 \end{align*}
  the desired regularity follows from \eqref{RegA}.
\end{proof}

For later purposes we decompose the solution operator found in Proposition~\ref{P1} as a sum of two operators. 
This decomposition   is very useful   because $\oo_2(f)[h] $ can be viewed as a lower order term, while the highest order term $(\oo_1(f)[h])'$ appears as a derivative, and  this  enables us to use integration by parts 
in the arguments that follow (see the proof of Theorem~\ref{TK1}).

\begin{prop}\label{P1d} Given $f\in H^2(\R) $  and $h\in H^3(\R)$, let  
\begin{align*}
 \oo_1(f)[h]&:=(1+a_\mu\bA(f))^{-1}\Big[\frac{h''}{(1+f'^2)^{3/2}}\Big],\\[1ex]
 \oo_2(f)[h]&:=(1+a_\mu\bA(f))^{-1}\big[-\Theta h'+a_\mu T_{0,{\rm lot}}(f)[\oo_1(f)[h]]\big],
\end{align*}
where the mapping $T_{0,{\rm lot}} $ is defined in \eqref{0lot}. 
Then:
\begin{itemize}
 \item[$(i)$] $\oo_1\in C^\omega(H^2(\R), \kL(H^3(\R), H^1(\R)))$ and $  \oo_2\in C^\omega(H^2(\R), \kL(H^3(\R), L_2(\R)));$
 \item[$(ii)$] \[\oo(f)=\frac{d}{dx}\circ\oo_1(f)+\oo_2(f); \]
  \item[$(iii)$] Given $\tau\in (1/2,1)$, there exists a constant $C$ such that 
\begin{align}
\|\oo_1(f)[h]\|_2\leq C\|h\|_{H^2}\qquad\text{and}\qquad  \|\oo_1(f)[h]\|_{H^{\tau}}+\|\oo_2(f)[h]\|_2\leq C\|h\|_{H^{2+\tau}} \label{p1} 
\end{align}
for all $h\in H^3(\R)$.
\end{itemize}

\end{prop}
\begin{proof}
That $\oo_1$ is well-defined and $\oo_1\in C^\omega(H^2(\R), \kL(H^3(\R), H^1(\R))) $ follows from  Proposition~\ref{P:I2}, \eqref{RegA}, and the property
\begin{align*}
  \left[f\mapsto \Big[h\mapsto\frac{h''}{(1+f'^2)^{3/2}}\Big]\right]\in C^\omega(H^2(\R),\kL(H^3(\R),H^1(\R))).
 \end{align*}
 Moreover,   in view of \eqref{DERR}, we get that 
\begin{align}
(1+a_\mu\bA(f))[(\oo_1(f)[h])']=\Big(\frac{h''}{(1+f'^2)^{3/2}}\Big)'-a_\mu T_{0,{\rm lot}}(f)[\oo_1(f)[h]],\label{DEE}
\end{align}
 and therewith  $\oo(f)[h]-(\oo_1(f)[h])'=\oo_2(f)[h] $ for all $h\in H^3(\R)$.
Invoking Proposition~\ref{P1}, we are left  to establish \eqref{p1}.
The estimates for $ \oo_1(f)[h] $ follow from  \eqref{DI} and \eqref{FE} via complex interpolation, cf. \eqref{CI}.  
Finally, the estimate for the term $\|\oo_2(f)[h]\|_2$ is a consequence of \eqref{DI}, \eqref{0lote}, and of the  estimate  for $\|\oo_1(f)[h]\|_{H^{\tau}}.$
\end{proof}

Appealing to Proposition~\ref{P1}, we  now   formulate the original  system \eqref{P}, after rescaling  the time appropriately, as a quasilinear evolution problem for $f$ only, that is
\begin{equation}\label{Pest1}
  \p_t f= \Phi_{\sigma}(f)[f],\quad t>0,\qquad f(0)=f_0,
\end{equation}
where $\Phi_{\sigma}:H^2(\R)\to \kL(H^3(\R), L_2(\R))$ is the mapping
\begin{align*}
 \Phi_{\sigma}(f)[h] :=\bB(f)[\ov\omega(f)[h]]
\end{align*}
for $f\in H^2(\R)$ and $h\in H^3(\R)$, and $\bB(f)$ is the linear operator  defined in \eqref{oper}.
The properties \eqref{regB} and \eqref{RegO} imply that
\begin{align}\label{reg2}
  \Phi_{\sigma}\in C^{\omega}  (H^2(\R),\kL(H^3(\R), L_2(\R))).
\end{align}
 
\medskip

\paragraph{\bf The generator property}
We now choose an arbitrary  function $f\in H^2(\R) $ which is is kept fixed in the following. 
Our next task is to prove that  $\Phi_{\sigma}(f)$, considered as an unbounded operator  in $L_2(\R)$ with definition domain $H^3(\R)$, is  the generator of a strongly continuous
and analytic semigroup in $\kL(L_2(\R))$, that is
\begin{align}\label{GP}
  -\Phi_{\sigma}(f)\in \kH  (H^3(\R), L_2(\R)),
\end{align}
see \cite{Am95, L95} for several characterizations of such type of operators.
To this end, we write  the operator $\Phi_{\sigma}(f)$ as a sum
\begin{align*} 
  \Phi_{\sigma}(f) =\Phi_{\sigma,1}(f)+\Phi_{\sigma,2}(f),  
\end{align*}
where
\begin{align*}
 \Phi_{\sigma,1}(f)[h] &:=B_{0,1}(f)[(\oo_1(f)[h])'] +f'B_{1,1}(f)[f,(\oo_1(f)[h])'], \\[1ex]
   \Phi_{\sigma,2}(f)[h] &:=\bB(f)[\oo_2(f)[h]]  
\end{align*}
for   $h\in H^3(\R)$, and where $\oo_i(f)$, $i=1,2,$ are the operators introduced in Proposition~\ref{P1d}.
For the particular choice $\tau=3/4$ in Proposition~\ref{P1d}, it follows that  $\Phi_{\sigma,2}(f)\in\kL(H^{11/4}(\R), L_2(\R)).$ This property together with $[L_2(\R), H^3(\R)]_{11/12}=H^{11/4}(\R),$ cf. \eqref{CI}, 
and \cite[Proposition~2.4.1]{L95} enables us to view $\Phi_{\sigma,2}(f)$ as a lower order perturbation.
Hence, our task reduces to establishing the generator property for the leading order term  $\Phi_{\sigma,1}(f).$
The proof of this property is  technical and  is based on an approach followed previously in \cite{E94,   ES95,  ES97} in the context of spaces of continuous functions, and refined recently in \cite{EMW15, M16x}.

In order to proceed we choose for each $\e\in(0,1)$ a so-called finite $\e$-localization family, that is  a family  
\[\{\pi_j^\e\,:\, -N+1\leq j\leq N\}\subset  C^\infty_0(\R,[0,1]),\]
with $N=N(\e)\in\N $ sufficiently large, such that
\begin{align}
\bullet\,\,\,\, \,\,  & \text{$ \supp \pi_j^\e $ is an interval of length less or equal $\e$ for all $|j|\leq N-1$;}
\label{i}\\[1ex]
\bullet\,\,\,\, \,\, &\text{$ \supp \pi_{N}^\e\subset(-\infty,-x_{N}]\cup [x_N,\infty)$  and $x_{N}\geq\e^{-1}$;}
\label{ii}\\[1ex]
\bullet\,\,\,\, \,\, &\text{ $\supp \pi_j^\e\cap\supp \pi_l^\e=\emptyset$ if $[|j-l|\geq2, |j|, |l|\in\{0,\dots, N-1\}]$ or $[|l|\leq N-2, j=N];$} \label{iii}\\[1ex]
\bullet\,\,\,\, \,\, &\text{ $\sum_{j=-N+1}^N(\pi_j^\e)^2=1$ and $\|(\pi_j^\e)^{(k)}\|_\infty\leq C\e^{-k}$ for all $ k\in\N, -N+1\leq j\leq N$.}\label{v} 
\end{align} 
 Such $\e$-localization families can be easily constructed.
Furthermore, we choose   a second family   \[\{\chi_j^\e\,:\, -N+1\leq j\leq N\}\subset  C^\infty_0(\R,[0,1]) \] with the following properties
\begin{align}
\bullet\,\,\,\, \,\,  &\text{$\chi_j^\e=1$ on $\supp \pi_j^\e$;}\label{c1}\\[1ex]
\bullet\,\,\,\, \,\,  &\text{$ \supp \chi_j^\e$ is an interval with $|\supp\chi_j^\e|\leq 3\e$ for $|j|\leq N-1$;}\label{c2}\\[1ex]
\bullet\,\,\,\, \,\,  &\text{$\supp\chi_N^\e\subset [|x|\geq x_N-\e]$.} \label{c3}
\end{align}

The following remark is a simple exercise.
\begin{rem}\label{R:3}
 Given  $k\in\N$ and a finite $\e$-localization family  $\{\pi_j^\e\,:\, -N+1\leq j\leq N\}$, the mapping
 \[
\Big[h\mapsto \sum_{j=-N+1}^ N\|\pi_j^\e h\|_{H^k}\Big]:H^k(\R)\to[0,\infty)
\]
defines a norm on $H^k(\R) $ which is equivalent to the standard $H^k$-norm. 
\end{rem}

Let us now introduce the continuous path
\[
[\tau\mapsto \Phi_{\sigma,1}(\tau f)]:[0,1]\to \kL(H^3(\R),L_2(\R))
\]
which connects  the operator $\Phi_{\sigma,1}(f)$ with   $\Phi_{\sigma,1}(0)$.
Recalling Proposition~\ref{P1d}, we have the following identity
\begin{align*}
 \Phi_{\sigma,1}(0)[h]=\pi H[(\oo_1(0)[h])']= \pi H[h''']=-\pi (\p_x^4)^{3/4}[h],
\end{align*}
where   $(\p_x^4)^{3/4}$ denotes the Fourier multiplier with symbol $[\xi\mapsto |\xi|^3].$
The following result   shows that  the operator $\Phi_{\sigma,1}(\tau f)$ can be approximated, in a sense to be made precise below,
by  Fourier multipliers $c(\p_x^4)^{3/4},$ where $c$ denotes negative constants.  

\begin{thm}\label{TK1} 
Let $f\in H^2(\R)$ and    $\mu>0$ be given. Then, there exist $\e\in(0,1)$, a finite $\e$-locali\-za\-tion family  $\{\pi_j^\e\,:\, -N+1\leq j\leq N\} $, a constant $K=K(\e)$, 
and for each  $ j\in\{-N+1,\ldots,N\}$ and $\tau\in[0,1]$ there 
exist bounded operators $$\bA_{ j,\tau}\in\kL(H^3(\R), L_2(\R))$$
 such that 
 \begin{equation}\label{DEK1}
  \|\pi_j^\e \Phi_{\sigma,1}(\tau f)[h]-\bA_{j,\tau}[\pi^\e_j h]\|_{2}\leq \mu \|\pi_j^\e h\|_{H^3}+K\|  h\|_{H^{11/4}}
 \end{equation}
 for all $ j\in\{-N+1,\ldots,N\}$, $\tau\in[0,1],$ and  $h\in H^3(\R)$. The operators $\bA_{j,\tau}$ are defined  by 
  \begin{align*} 
 \bA_{j,\tau }:=- \frac{   \pi}{(1+\tau^2 f'^2(x_j^\e) )^{3/2}}   (\p_x^4 )^{3/4}, \qquad |j|\leq N-1, 
 \end{align*}
 where $x_j^\e\in \supp  \pi_j^\e$, respectively
 \begin{align*} 
 \bA_{N,\tau }:= -  \pi (\p_x^4)^{3/4}. 
 \end{align*}
\end{thm}
\begin{proof} 
We first pick   a finite $\e$-localization family  $\{\pi_j^\e\,:\, -N+1\leq j\leq N\} $  and an associated family  $\{\chi_j^\e\,:\, -N+1\leq j\leq N\} $, 
with $\e\in(0,1)$ to be  fixed later on in the proof.
It is convenient to write
\begin{align}\label{LL2}
 \Phi_{\sigma,1}(\tau f)[h] =\sum_{l=0}^1 f_{l,\tau}'B_{1,1}(\tau f)[f_{l,\tau},(\oo_1(\tau f)[h])'],\qquad h\in H^3(\R),
\end{align}
where, for $l\in\{0,1\}$, $f_{l,\tau}$ denotes the Lipschitz function
\[
f_{l,\tau}:=(1-l)\tau f+l{\, \rm id}_\R.
\] 
We further let 
\[
\bA^l_{j,\tau} :=- \pi \frac{ (f_{l,\tau}'(x_j^\e) )^2}{(1+\tau^2 f'^2(x_j^\e)  )^{5/2}}  (\p_x^4 )^{3/4}= \frac{ (f_{l,\tau}'(x_j^\e) )^2}{(1+\tau^2 f'^2(x_j^\e)  )^{5/2}}B_{0,1}(0)\circ\p_x^3\qquad\text{for $|j|\leq N-1$},
\]
respectively
\[
\bA^l_{N,\tau} :=- \pi  l^2   (\p_x^4 )^{3/4}=l^2B_{0,1}(0)\circ\p_x^3.
\]
In the following we denote by $C$ constants which are
independent of $\e$ (and, of course, of $h\in H^3(\R)$, $\tau\in [0,1]$, $l\in\{0,1\}$, and $j \in \{-N+1, \ldots, N\}$) and the constants  denoted by $K$ may depend only upon $\e.$
To establish \eqref{DEK1} we first consider  the case $|j|\leq N-1$.\medskip

\noindent{\em The case $|j|\leq N-1$.\,\,} 
Given $l\in\{0,1\}$ and $|j|\leq N-1$, we write
\begin{align}\label{LL2a}
 \pi_j^\e f_{l,\tau}'B_{1,1}(\tau f)[f_{l,\tau},(\oo_1(\tau f)[h])']-\bA^l_{j,\tau}[\pi^\e_j h]:=T_1[h]+T_2[h]+T_3[h],
\end{align}
where 
\begin{align*}
T_1[h]&:= \pi_j^\e f_{l,\tau}'B_{1,1}(\tau f)[f_{l,\tau},(\oo_1(\tau f)[h])'] - f_{l,\tau}'(x_j^\e)B_{1,1}(\tau f)  [f_{l,\tau},\pi^\e_j(\oo_1(\tau f)[h])' ], \\[1ex]
T_2[h]&:= f_{l,\tau}'(x_j^\e)B_{1,1}(\tau f)[f_{l,\tau},\pi^\e_j(\oo_1(\tau f)[h])']-\frac{(f_{l,\tau}'(x_j^\e))^2}{1+\tau^2 f'^2(x_j^\e)}B_{0,1}(0)  [\pi^\e_j(\oo_1(\tau f)[h])' ],\\[1ex]
T_3[h]&:= \frac{(f_{l,\tau}'(x_j^\e))^2}{1+\tau^2 f'^2(x_j^\e)}B_{0,1}(0)  [\pi^\e_j(\oo_1(\tau f)[h])' ]-\bA^l_{j,\tau}[\pi^\e_j h].
\end{align*}
To begin, we compute, by using   the fact that $\chi_j^\e=1$ on $\supp\pi_j^\e$, that
\begin{align}\label{LL2w}
 T_1[h]=&\chi_j^\e(f_{l,\tau}'-f_{l,\tau}'(x_j^\e))B_{1,1}(\tau f)  [f_{l,\tau},\pi^\e_j(\oo_1(\tau f)[h])' ]+T_{11}[h],
\end{align}
where, using integration by parts, we may reexpress $T_{11}[h]$ in the following way
\begin{align*}
 T_{11}[h]&=f_{l,\tau}'B_{1,1}(\tau f)  [f_{l,\tau},(\pi_j^\e)'\oo_1(\tau f)[h] ]+f_{l,\tau}'B_{1,1}(\tau f)  [\pi_j^\e,f_{l,\tau}'\oo_1(\tau f)[h]]\\[1ex]
 &\hspace{0.424cm}-2f_{l,\tau}'B_{2,1}( \tau f)  [\pi_j^\e,f_{l,\tau},\oo_1(\tau f)[h]]-2\tau^2 f_{l,\tau}'B_{3,2}(\tau f, \tau f)  [\pi_j^\e,f_{l,\tau},f,f' \oo_1(\tau f)[h]]\\[1ex]
 &\hspace{0.424cm}+2\tau^2 f_{l,\tau}'B_{4,2}(\tau f, \tau f)  [\pi_j^\e,f_{l,\tau},f,f, \oo_1(\tau f)[h]]\\[1ex]
 &\hspace{0.424cm}+(f_{l,\tau}'(x_j^\e)-f_{l,\tau}')(1-\chi_j^\e)B_{1,1}(\tau f)  [f_{l,\tau},(\pi_j^\e)'\oo_1(\tau f)[h]]\\[1ex]
 &\hspace{0.424cm}+(f_{l,\tau}'(x_j^\e)-f_{l,\tau}')B_{1,1}(\tau f)  [\chi_j^\e,\pi_j^\e f_{l,\tau}'\oo_1(\tau f)[h]]\\[1ex]
 &\hspace{0.424cm}-2(f_{l,\tau}'(x_j^\e)-f_{l,\tau}')B_{2,1}( \tau f)  [\chi_j^\e,f_{l,\tau},\pi_j^\e \oo_1(\tau f)[h]]\\[1ex]
 &\hspace{0.424cm}-2\tau^2 (f_{l,\tau}'(x_j^\e)-f_{l,\tau}')B_{3,2}(\tau f, \tau f)  [\chi_j^\e,f_{l,\tau},f,\pi_j^\e f' \oo_1(\tau f)[h]]\\[1ex]
 &\hspace{0.424cm}+2\tau^2 (f_{l,\tau}'(x_j^\e)-f_{l,\tau}')B_{4,2}(\tau f, \tau f)  [\chi_j^\e,f_{l,\tau},f,f,\pi_j^\e \oo_1(\tau f)[h]].
\end{align*}
 Combining Lemma~\ref{L:99a}~$(i)$ and Proposition~\ref{P1d}~$(iii)$, we see that
 \begin{align*}
  \|T_{11}[h]\|_{2}\leq   K\|h\|_{H^2},
\end{align*}
and  therewith
 \begin{align}
  \|T_{1}[h]\|_{2}\leq   C\|\chi_j^\e(f_{l,\tau}'-f_{l,\tau}'(x_j^\e))  \|_\infty\|\pi^\e_j(\oo_1(\tau f)[h])' \|_2+ K\|h\|_{H^2}.\label{LL2c}
\end{align}
We next estimate the term $\|\pi^\e_j(\oo_1(\tau f)[h])' \|_2$.
In view of \eqref{DEE} and using integration by parts, we have
\begin{align}
 (1+a_\mu \bA(\tau f))[\pi_j^\e(\oo_1(\tau f)[h])']&=\frac{ \pi_j^\e h'''}{(1+\tau^2f'^2)^{3/2}}-\frac{3\tau^2\pi_j^\e f'f'' h''}{(1+\tau^2f'^2)^{5/2}}-a_\mu \pi_j^\e T_{0,{\rm lot}}(\tau f)[\oo_1(\tau f)[h]]\nonumber\\[1ex]
 &\hspace{0.424cm}-\frac{ a_\mu }{\pi}T_{1j,{\rm lot}}(\tau f)[\oo_1(\tau f)[h]],\label{LL2e}
\end{align}
where, given $f\in H^2(\R), $   $\oo\in H^1(\R)$, and $j\in\{-N+1,\ldots, N\}$, we set 
\begin{align}
 T_{1j,{\rm lot}}(f)[\oo]&:=\pi\big(\pi_j^\e\bA(f)[\oo']-\bA(f)[\pi_j^\e\oo']\big)\nonumber\\[1ex]
 &\phantom{:}=f'\Big(B_{0,1}(f)[(\pi_j^\e)'\oo]-B_{1,1}( f)[\pi_j^\e, \oo ] -2 B_{2,2}(f, f)[\pi_j^\e,f,f'\oo] +2B_{3,2}(f ,f)[\pi_j^\e,f,f,\oo]\Big)\nonumber\\[1ex]
&\hspace{0.424cm}-B_{1,1}(f)[\pi_j^\e,f'\oo]+2B_{2,1}(f)[\pi_j^\e,f, \oo ] -B_{1,1}(f)[f,(\pi_j^\e)'\oo] \nonumber\\[1ex]
&\hspace{0.424cm}+2 B_{3,2}(f, f)\big)[\pi_j^\e,f,f,f'\oo]-2B_{4,2}(f, f)\big)[\pi_j^\e,f,f,f,\oo],\label{T1j}
\end{align}
the last identity following by using integration by parts.
In view of \eqref{LL2e}, it follows from \eqref{DI},  \eqref{0lote} and \eqref{p1} (with $\tau=3/4$) that 
\begin{align}
  \|\pi_j^\e(\oo_1(\tau f)[h])'\|_{2}\leq  C\|\pi^\e_j h \|_{H^3}+ K\|h\|_{H^{11/4}}.\label{LL2f}
\end{align}
The estimate \eqref{LL2f} is clearly valid also for $j=N$.
Choosing $\e$ sufficiently small, it follows from \eqref{LL2c} and  \eqref{LL2f}  that 
\begin{align}
  \|T_{1}[h]\|_{2}\leq  \frac{\mu}{3}\|\pi^\e_j h \|_{H^3}+ K\|h\|_{H^{11/4}},\label{LL2aa}
\end{align}
provided that $\e$ is sufficiently small.  

We now consider the term $T_2[h]$ which is decomposed as follows 
\begin{align*}
 T_2[h]=&f_{l,\tau}'(x_j^\e)\big(B_{1,1}(\tau f)[f_{l,\tau},\pi^\e_j(\oo_1(\tau f)[h])']-B_{1,1}(\tau f'(x_j^\e){\rm id}_\R)[f_{l,\tau}(x_j^\e){\rm id}_\R,\pi^\e_j(\oo_1(\tau f)[h])']\big)\\[1ex]
 =& f_{l,\tau}'(x_j^\e)  T_{21}[h]-\frac{\tau^2f_{l,\tau}'^2(x_j^\e)}{1+\tau^2f'^2(x_j^\e)} T_{22}[h],
\end{align*}
 with
\begin{align*}
 T_{21}[h]&:=B_{1,1}(\tau f)[f_{l,\tau}-f_{l,\tau}'(x_j^\e){\rm id}_\R ,\pi_j^\e(\oo_1(\tau f)[h])'],\\[1ex]
  T_{22}[h]&:=B_{2,1}(\tau f)[f-f'(x_j^\e){\rm id}_\R,f+f'(x_j^\e){\rm id}_\R,\pi_j^\e(\oo_1(\tau f)[h])'].
\end{align*}
We  first estimate $T_{21}[h]$. Integrating by parts, we get
\begin{align}
 \|T_{21}[h]\|_2&\leq\|\chi_j^\e B_{1,1}(\tau f)[f_{l,\tau}-f_{l,\tau}'(x_j^\e){\rm id}_\R ,\pi_j^\e(\oo_1(\tau f)[h])']\|_2\nonumber\\[1ex]
 &\hspace{0.424cm}+\Big\|\PV\int_\R\frac{[\delta_{[\cdot,y]} (f_{l,\tau}-f_{l,\tau}'(x_j^\e){\rm id}_\R)/y][\delta_{[\cdot, y]}\chi_j^\e/y]}{1+\tau^2\big( \delta_{[\cdot,y]} f/y\big)^2}(\pi^\e_j(\oo_1(\tau f)[h])')(\cdot-y)\,dy\Big\|_2\nonumber\\[1ex]
&\leq \|\chi_j^\e B_{1,1}(\tau f)[f_{l,\tau}-f_{l,\tau}'(x_j^\e){\rm id}_\R ,\pi_j^\e(\oo_1(\tau f)[h])']\|_2+ K\|h\|_{H^2}.\label{LL3a}
\end{align}
Furthermore, letting $F_{l,\tau,j}\in C(\R)$ denote the Lipschitz function satisfying $F_{l,\tau,j}=f_{l,\tau}$ on $\supp \chi_j^\e$ and $F_{l,\tau,j}'=f_{l,\tau}'(x_j^\e)$ on $\R\setminus \supp \chi_j^\e,$
we have
\begin{align*}
 \chi_j^\e B_{1,1}(\tau f)[f_{l,\tau}-f_{l,\tau}'(x_j^\e){\rm id}_\R ,\pi_j^\e(\oo_1(\tau f)[h])']=\chi_j^\e B_{1,1}(\tau f)[F_{l,\tau,j}-f_{l,\tau}'(x_j^\e){\rm id}_\R ,\pi_j^\e(\oo_1(\tau f)[h])'],
\end{align*}
and   \eqref{LL2f} combined with  \eqref{LL3a}  leads us to
\begin{align*}
 \|T_{21}[h]\|&\leq C\|f_{l,\tau}'-f_{l,\tau}'(x_j^\e)\|_{L_\infty(\supp \chi_j^\e)}\|\pi_j^\e(\oo_1(\tau f)[h])'\|_2+ K\|h\|_{H^2}\\[1ex]
 &\leq\frac{\mu}{6}\|\pi^\e_j h \|_{H^3}+ K\|h\|_{H^{11/4}},
\end{align*}
provided that $\e$ is sufficiently small.
Since $T_{22}[h]$ can be estimated in a similar way, we  conclude that
\begin{align}
  \|T_{2}[h]\|_{2}\leq   \frac{\mu}{3}\|\pi^\e_j h \|_{H^3}+ K\|h\|_{H^{11/4}}\label{LL2aaa}
\end{align}
if $\e$ is chosen sufficiently small.

We now turn to  $T_3[h]$ and  note that 
\begin{align*}
 T_3[h] = \frac{(f_{l,\tau}'(x_j^\e))^2}{1+\tau^2 f'^2(x_j^\e)} B_{0,1}(0)  \Big[\pi^\e_j(\oo_1(\tau f)[h])' -\frac{1}{(1+f'^2(x_j^\e))^{3/2}}(\pi^\e_j h)'''\Big],   
\end{align*}
and therefore
\begin{align}\label{L4a}
 \|T_3[h]\|_2\leq C\Big\|\pi^\e_j(\oo_1(\tau f)[h])' -\frac{1}{(1+f'^2(x_j^\e))^{3/2}}\pi^\e_j h'''\Big\|_2+K\|h\|_{H^2}.
\end{align}
Recalling \eqref{LL2e}, we have
\begin{align*}
\pi_j^\e(\oo_1(\tau f)[h])'-\frac{\pi^\e_j h'''}{((1+f'^2(x_j^\e)))^{3/2}}&= \Big[\frac{1}{(1+\tau^2f'^2)^{3/2}}-\frac{1}{(1+f'^2(x_j^\e))^{3/2}}\Big]\pi^\e_j h'''\\[1ex]
&\hspace{0.424cm}-\frac{3\tau^2\pi_j^\e f'f'' h''}{(1+\tau^2f'^2)^{5/2}}-a_\mu \pi_j^\e T_{0,{\rm lot}}(\tau f)[\oo_1(\tau f)[h]]\\[1ex]
 &\hspace{0.424cm}-\frac{ a_\mu }{\pi}T_{1j,{\rm lot}}(\tau f)[\oo_1(\tau f)[h]]-a_\mu \bA(\tau f)[\pi_j^\e(\oo_1(\tau f)[h])'],
\end{align*}
and \eqref{L4a} combined with \eqref{0lote} and \eqref{p1} (with $\tau=3/4$)  yields 
\begin{align}\label{L4b}
 \|T_3[h]\|_2\leq \|\bA(\tau f)[\pi_j^\e(\oo_1(\tau f)[h])']\|_2+\frac{\mu}{6}\|\pi^\e_j h \|_{H^3}+ K\|h\|_{H^{11/4}},
\end{align}
provided that $\e$ is sufficiently small.
We are left with the term
\begin{align*} 
 \bA(\tau f)[\pi_j^\e(\oo_1(\tau f)[h])']&=\tau\big(f'B_{0,1}(\tau f)[\pi_j^\e(\oo_1(\tau f)[h])'] -B_{1,1}(\tau f)[f,\pi_j^\e(\oo_1(\tau f)[h])']\big)\nonumber\\[1ex]
 &=\tau(T_{31}[h]+T_{32}[h]),
\end{align*}
where
\begin{align*}
 T_{31}[h]&:=(f'-f'(x_j^\e))B_{0,1}(\tau f)[\pi_j^\e(\oo_1(\tau f)[h])'],\\[1ex]
 T_{32}[h]&:=B_{1,1}(\tau f)[f'(x_j^\e){\rm id}_\R-f,\pi_j^\e(\oo_1(\tau f)[h])'].
\end{align*}
The term $T_{32}[h]$ may be estimated in a  similar way as the term $T_{21}[h]$ above, while integrating by parts, we obtain, similarly as in the study of $T_1[h]$, the following estimate 
\begin{align*}
 \|T_{31}[h]\|_2&\leq\|\chi_j^\e(f'-f'(x_j^\e))B_{0,1}(\tau f)[\pi_j^\e(\oo_1(\tau f)[h])']\|_2\\[1ex]
 &\hspace{0.424cm}+\Big\|(f'-f'(x_j^\e))\PV\int_\R\frac{\delta_{[\cdot,y]}\chi_j^\e/y}{1+\tau^2\big( \delta_{[\cdot,y]} f/y\big)^2} (\pi^\e_j(\oo_1(\tau f)[h])')(\cdot-y) \,dy\Big\|_2\\[1ex]
 &\leq\|\chi_j^\e(f'-f'(x_j^\e))\|_\infty\|\pi_j^\e(\oo_1(\tau f)[h])']\|_2+K\|h\|_{H^2}.
\end{align*}
For   sufficiently small $\e$, \eqref{LL2f} leads us to 
\begin{align*} 
 \|T_{31}[h]\|_2+\|T_{32}[h]\|_2\leq \frac{\mu}{6}\|\pi^\e_j h \|_{H^3}+ K\|h\|_{H^{11/4}}
\end{align*}
and, together with \eqref{L4b}, we conclude  
\begin{align}\label{LL2aaaa}
 \|T_3[h]\|_2\leq  \frac{\mu}{3}\|\pi^\e_j h \|_{H^3}+ K\|h\|_{H^{11/4}}.
\end{align}
The desired estimate \eqref{DEK1} follows for $|j|\leq N-1$ from \eqref{LL2}, \eqref{LL2a}, \eqref{LL2aa}, \eqref{LL2aaa}, and  \eqref{LL2aaaa}.\medskip

\noindent{\em The case $j=N$.\,\,}    Similarly as in the previous case, we write
\begin{align}\label{LR0}
 \pi_N^\e f_{l,\tau}'B_{1,1}(f_{l,\tau},\tau f)[(\oo_1(\tau f)[h])']-\bA^l_{N,\tau}[\pi^\e_N h]:=S_1[h]+S_2[h]+S_3[h],
\end{align}
with
\begin{align*}
S_1[h]&:= \pi_N^\e f_{l,\tau}'B_{1,1}(\tau f)[f_{l,\tau},(\oo_1(\tau f)[h])'] - lB_{1,1}(\tau f)  [f_{l,\tau},\pi^\e_N(\oo_1(\tau f)[h])' ], \\[1ex]
S_2[h]&:= l B_{1,1}(\tau f)  [f_{l,\tau},\pi^\e_N(\oo_1(\tau f)[h])' ]-l^2B_{0,1}(0)  [\pi^\e_N(\oo_1(\tau f)[h])' ],\\[1ex]
S_3[h]&:= l^2B_{0,1}(0)  [\pi^\e_N(\oo_1(\tau f)[h])' ]-\bA^l_{N,\tau}[\pi^\e_N h].
\end{align*}
The estimates derived when studying $T_1[h]$ together with the fact that $f'$ vanishes at infinity imply that
\begin{align}
  \|S_{1}[h]\|_{2}\leq  \frac{\mu}{3}\|\pi^\e_N h \|_{H^3}+ K\|h\|_{H^2}\label{LR1}
\end{align}
for sufficiently small $\e$. 

If $l=0$, then $S_i=0,$ $i=2,3,$ and   we are left to consider the case $l=1$, when $f_{l,\tau}={\rm id}_\R $.
The relation \eqref{rell} implies that
\begin{align*}
 S_2[h]&= \big(B_{0,1}(\tau f) -B_{0,1}(0)\big)  [\pi^\e_N(\oo_1(\tau f)[h])' ]=-\tau^2B_{2,1}(\tau f) [f,f,\pi^\e_N(\oo_1(\tau f)[h])' ]\\[1ex]
 &=\tau^2\big(S_{21}[h]-S_{22}[h]\big),
\end{align*}
where
\begin{align*}
 S_{21}[h]&:= \tau^2\PV\int_\R\frac{\big(\delta_{[\cdot,y]} f/y\big)^2\big(\delta_{[\cdot,y]}\chi_N^\e/y\big)}{1+\tau^2\big( \delta_{[\cdot,y]} f/y\big)^2}[\pi^\e_N(\oo_1(\tau f)[h])' ](\cdot-y)\, dy,\\[1ex]
 S_{22}[h]&:= \chi_N^\e B_{2,1}( \tau f) [f,f, \pi^\e_N(\oo_1(\tau f)[h])' ].
\end{align*}
Integrating by parts we get
\begin{align*}
  \|S_{21}[h]\|_{2}\leq   K\|h\|_{H^2}.
\end{align*}
To deal with $S_{22}[h]$, we let $F_N\in W^1_\infty(\R)$ denote  the function defined by
\[
F_N(x):=\left\{
\begin{array}{lll}
f(x)&,& |x|\geq x_N-\e, \\[1ex]
\cfrac{x+x_N-\e}{2(x_N-\e)}f(x_N-\e)+\cfrac{x_N-\e-x}{2(x_N-\e)}f(-x_N+\e)&,& |x|\leq x_N-\e.
\end{array}
\right.
\]
Since $f, f'\in C_0(\R)$, the relation \eqref{ii} implies that $\|F_N'\|_\infty\to0 $ for $\e\to0$.
Moreover, recalling \eqref{LL2f}, we find for $\e$   sufficiently small that
\begin{align*}
 \|S_{22}[h]\|_2&=\| \chi_N^\e B_{2,1}(\tau f) [F_N,F_N,\pi^\e_N(\oo_1(\tau f)[h])' ]\|_2\leq C\|F_N'\|_\infty^2\|\pi^\e_N(\oo_1(\tau f)[h])' \|_2\\[1ex]
 &\leq\frac{\mu}{3}\|\pi^\e_N h \|_{H^3}+ K\|h\|_{H^{11/4}},
\end{align*}
and therewith
\begin{align}
  \|S_{2}[h]\|_{2}\leq  \frac{\mu}{3}\|\pi^\e_N h \|_{H^3}+ K\|h\|_{H^{11/4}}.\label{LR2}
\end{align}

Finally, it holds that
\begin{align}\label{L5a}
 \|S_3[h]\|_2&=\|B_{0,1}(0)[\pi^\e_N(\oo_1(\tau f)[h])' - (\pi^\e_N h)''']\|_2\nonumber\\[1ex]
 &\leq C\|\pi^\e_N(\oo_1(\tau f)[h])' - \pi^\e_N h'''\|_2+K\|h\|_{H^2},
\end{align}
 and, since $f$ vanishes at infinity,  the  arguments used to derive \eqref{L4b} show that
 \begin{align}\label{L5b}
\|\pi_N^\e(\oo_1(\tau f)[h])'-\pi^\e_N h'''\|_2\leq\|\bA(\tau f)[\pi_N^\e(\oo_1(\tau f)[h])']\|_2+\frac{\mu}{6}\|\pi^\e_N h \|_{H^3}+ K\|h\|_{H^{11/4}},
\end{align}
provided that $\e$ is sufficiently small.
Furthermore, the first term on the right-hand side of \eqref{L5b} is decomposed as follows
\begin{align*}
 \bA(\tau f)[\pi_N^\e(\oo_1(\tau f)[h])']&=\tau\big(f'B_{0,1}(\tau f)[\pi_N^\e(\oo_1(\tau f)[h])'] -B_{1,1}(\tau f)[f,\pi_N^\e(\oo_1(\tau f)[h])']\big)\\[1ex]
 &=\tau\chi_N^\e\big( f'B_{0,1}(\tau f)[\pi_N^\e(\oo_1(\tau f)[h])'] -B_{1,1}(\tau f)[F_N,\pi_N^\e(\oo_1(\tau f)[h])']\big)\\[1ex]
 &\hspace{0.424cm}-\tau f'\PV\int_\R\frac{\delta_{[\cdot,y]}\chi_N^\e/y}{1+\tau^2\big( \delta_{[\cdot,y]} f/y\big)^2} (\pi^\e_N(\oo_1(\tau f)[h])')(\cdot-y) \,dy\\[1ex]
 &\hspace{0.424cm}+\tau \PV\int_\R\frac{\big(\delta_{[\cdot,y]}f/y\big)\big(\delta_{[\cdot,y]}\chi_N^\e/y\big)}{1+\tau^2\big( \delta_{[\cdot,y]} f/y\big)^2} (\pi^\e_N(\oo_1(\tau f)[h])')(\cdot-y) \,dy,
\end{align*}
where $F_N$  is the function introduced when considering $S_{22}[h]$.
Integrating by parts the integral terms, we infer from Lemma~\ref{L:99a}~$(i)$, \eqref{p1}, \eqref{LL2f}, and  the fact that $f'$ vanishes at infinity that
\begin{align*} 
 \|\bA(\tau f)[\pi_N^\e(\oo_1(\tau f)[h])']\|_2 \leq \frac{\mu}{6}\|\pi^\e_j h \|_{H^3}+ K\|h\|_{H^{11/4}}
\end{align*}
for $\e$ sufficiently small. The latter estimate together with \eqref{L5a} and \eqref{L5b} implies that 
\begin{align}\label{LR3}
 \|S_3[h]\|_2\leq  \frac{\mu}{3}\|\pi^\e_j h \|_{H^3}+ K\|h\|_{H^{11/4}}
\end{align}
if $\e$ is sufficiently small. 
Summarizing, for $j=N,$ the desired estimate \eqref{DEK1} follows from \eqref{LR0}, \eqref{LR1}, \eqref{LR2}, \eqref{LR3}.
This completes the proof.
\end{proof}

The Fourier multipliers $\bA_{j,\tau}$ found in Theorem~\ref{TK1} are elements of the  family of unbounded operators $\{\bA_{ x_0,\tau}\,:\, \tau\in[0,1], x_0\in\R\}$, where 
\[
\bA_{x_0,\tau}:=-\frac{\pi}{(1+\tau^2 f'^2(x_0))^{3/2}}(\p_x^4)^{3/4}.
\]
Each operator $\bA_{x_0,\tau}$ is the  generator of a strongly continuous analytic semigroup in $\kL(L_2(\R)).$
Moreover,  it is not difficult to prove (see for example  the proof  of \cite[Proposition~6.3]{M16x}) that there exists a constant  $\kappa_0\geq1$  such that  
  \begin{align}\label{13K}
&\lambda-\bA_{x_0,\tau}\in{\rm Isom}(H^3(\R), L_2(\R)),\\[1ex]
\label{14K}
& \kappa_0\|(\lambda-\bA_{x_0,\tau})[h]\|_{2}\geq  |\lambda|\cdot\|h\|_{2}+\|h\|_{H^3}
\end{align}
for all $x_0\in\R$, $\tau\in[0,1],$   $\lambda\in\C$ with $\re \lambda\geq 1$, and $h\in H^3(\R)$.
Combining these properties with Theorem~\ref{TK1}, we obtain the following generation result.
\begin{thm}\label{TK2}
Given $f\in H^2(\R)$, it holds that 
 \begin{align*} 
  -\Phi_{\sigma}(f)\in \kH(H^3(\R), L_2(\R)).
 \end{align*}
\end{thm}
\begin{proof}
 The proof is similar to that of \cite[Theorem~6.3]{M16x}, so that we only present the main steps. 
 Using the inequality   \eqref{14K}, Remark \ref{R:3}, and   Theorem~\ref{TK1},  we may find constants $\omega>1$ and $\kappa\geq1$ such that 
 \begin{align}\label{14K*}
\kappa\|(\lambda-\Phi_{\sigma,1}(\tau f))[h]\|_{2}\geq  |\lambda|\cdot\|h\|_{2}+\|h\|_{H^3}
\end{align}
for all $\tau\in[0,1]$, $\lambda\in\C$ with $\re \lambda\geq \omega$, and $h\in H^3(\R)$.
Additionally, since $\omega-\Phi_{\sigma,1}(0)=\omega+\pi(\p_x^4)^{3/4}$, the relation \eqref{13K} shows that $\omega-\Phi_{\sigma,1}(0)\in{\rm Isom}(H^3(\R), L_2(\R)),$ and \eqref{14K*} together with  the method of continuity yields 
 \begin{align}\label{13K*}
\omega-\Phi_{\sigma,1}(f)\in{\rm Isom}(H^3(\R), L_2(\R)).
\end{align}
Since $\Phi_{\sigma,2}(f)$ is a lower order perturbation, the claim follows from \eqref{14K*} and \eqref{13K*}. 
\end{proof}

We now come to the proof of our  first main result, which is mainly based on the abstract theory for  quasilinear parabolic problems   due to H. Amann, cf.  \cite[Section 12]{Am93}, and a recent  idea from  the proof of \cite[Theorem~1.3]{M16x}
where a parameter  trick, also used in   \cite{An90, ES96, PSS15},  is employed in a different manner.  

\begin{proof}[Proof of Theorem~\ref{MT1}]
 In virtue of \eqref{reg2} and of Theorem~\ref{TK2}, we have 
 \[
-\Phi_{\sigma}\in C^{\omega}(H^{2}(\R),\kH(H^3(\R), L_2(\R))).
\]
This relation together with the well-known interpolation property
\begin{align}\label{CI}
[H^{s_0}(\R),H^{s_1}(\R)]_\theta=H^{(1-\theta)s_0+\theta s_1}(\R),\qquad\theta\in(0,1),\, -\infty\leq s_0\leq s_1<\infty,
\end{align}
the choices $s\in(2,3)$, $\ov s=2$, $1>\alpha:=s/3>\beta:=2/3>0$, and the relations
 \[
  H^{2}(\R)=[L_2(\R), H^3(\R)]_{\beta} \qquad\text{and}\qquad H^s(\R)=[L_2(\R), H^3(\R)]_{\alpha},
 \]
enables us to use  abstract results  \cite{Am93, Am86, Am88}, see \cite[Theorem 1.5]{M16x}, to establish, for each  $f_0\in H^s(\mathbb{R})$, $s\in(2,3)$,  the existence of a unique classical solution  $f= f(\,\cdot\, ; f_0)$
to    \eqref{Pest1}   such that 
 \begin{equation*} 
 f\in C([0,T_+(f_0)),H^s(\mathbb{R}))\cap C((0,T_+(f_0)), H^3(\mathbb{R}))\cap C^1((0,T_+(f_0)), L_2(\mathbb{R}))
  \end{equation*}
and\begin{equation*} 
 f\in C^{(s-2)/3}([0,T],H^{2}(\mathbb{R})) \qquad\text{for all $T<T_+(f_0)$}.
  \end{equation*}

Concerning  the uniqueness claim, it suffices to  show  that any classical solution 
 \begin{equation*} 
 f\in C([0,\wt  T),H^s(\R))\cap C((0,\wt T), H^3(\R))\cap C^1((0,\wt T)), L_2(\R)),\qquad \wt T\in(0,\infty],
 \end{equation*}  
to \eqref{Pest1}   satisfies 
  \begin{equation} \label{T:EEE2}
   f\in C^\eta([0,T],H^{2}(\R))\qquad \text{for all $T\in(0,\wt T)$,}
 \end{equation}
with $\eta:=(s-2)/(s+1),$  see \cite[Theorem 1.5]{M16x}.
Let thus $T\in(0,\wt T)$ be fixed. 
The boundedness of $f:[0,T]\to H^s(\R)$, Theorem \ref{T:I1},  Proposition \ref{P:I2}, and an interpolation argument imply that
\begin{align*} 
\sup_{t\in[0,T]}\|(1+a_\mu \bA(f))^{-1}\|_{\kL(H^{r}(\R))}\leq C 
\end{align*}
for all $r\in(0,1)$ (in particular for $r=s-2$), and recalling that $\oo_1(f)[f]= (1+a_\mu \bA(f))^{-1}[\kappa(f)]$ we are led to
\begin{align}\label{LE2}
\sup_{t\in[0,T]}\|\oo_1(f)[f]\|_{H^{s-2}}\leq C 
\end{align}
Arguing as in the proof of \cite[Theorem 1.2]{M16x}, we obtain in view of \eqref{LE2}  that
\begin{align}\label{DEDE1}
\sup_{t\in(0,T]}\|\Phi_{\sigma,1}(f)[f]\|_{H^{-1}}\leq C.
\end{align}

In order to study the boundedness of 
\[
 \Phi_{\sigma,2}(f)[f]=-\Theta\bB(f)\big[(1+a_\mu \bA(f))^{-1}[f']\big]+a_\mu\bB(f)\big[(1+a_\mu \bA(f))^{-1}\big[T_{0,\rm lot}(f)[\oo_1(f)[f]]\big]\big],
\]
 we first note that
 \begin{align*} 
  \sup_{t\in[0,T]}\|(1+a_\mu \bA(f))^{-1}[f']\|_2\leq C,
 \end{align*}
 and Lemma \ref{L:99a} implies 
 \begin{align} \label{DEDE2}
  \sup_{t\in[0,T]}\|\bB(f)[(1+a_\mu \bA(f))^{-1}[f']]\|_2\leq C.
 \end{align}
To deal with the remaining term,  we may argue as   in the proof of \cite[Theorem 1.2]{M16x} to obtain, in virtue of  \eqref{LE2}, that
 \begin{align}\label{LE4}
  \sup_{t\in(0,T]}\| T_{0,\rm lot}(f)[\oo_1(f)[f]]\|_1\leq C.
 \end{align}
 Given  $t\in(0,T]$, it holds that  $f=f(t)\in H^3(\R)$ and by  Proposition \ref{P1d} $(iii)$ we deduce that $\oo_3(f):=(1+a_\mu \bA(f))^{-1}\big[T_{0,\rm lot}(f)[\oo_1(f)[f]]\big]\in L_2(\R)$ with   
\begin{align*}
  \int_\R T_{0,\rm lot}(f)[\oo_1(f)[f]]\varphi\, dx&=\int_\R\varphi(1+a_\mu \bA(f))[\oo_3(f)]\, dx=\int_\R\oo_3(f)(1+a_\mu (\bA(f))^*)[\varphi]\, dx
\end{align*}
 for all $\varphi\in H^1(\R)$, 
where $(\bA(f))^*\in\kL(L_2(\R))$ is the adjoint of $\bA(f) $, cf. Lemma \ref{L:99f}.
 Moreover, since 
 \[
 \sup_{t\in[0,T]}\|(1+a_\mu (\bA(f))^*)^{-1}\|_{\kL(H^1(\R))}\leq C,
 \]
the estimate \eqref{LE4} leads us to 
 \begin{align*}
  \Big|\int_\R \oo_3(f)\psi\, dx\Big|&=\Big|\int_\R T_{0,\rm lot}(f)[\oo_1(f)[f]](1+a_\mu (\bA(f))^*)^{-1}[\psi]\, dx\Big|\leq C\|\psi\|_{H^1} 
 \end{align*}
for all $\psi\in H^1(\R)$, hence
\begin{align}\label{LE5}
  \sup_{t\in(0,T]}\|  \oo_3 (f)\|_{H^{-1}}\leq C.
 \end{align}
Letting $(\bB(f))^*\in\kL(L_2(\R))$ denote  the adjoint of $ \bB(f)$,  it is not difficult to see that 
  \[
  (\bB(f))^*[\varphi]=-B_{0,1}(f)[\varphi]-B_{1,1}(f)[f, f'\varphi],\qquad \varphi\in L_2(\R),
  \]
  and  the Lemmas \ref{L:99a}-\ref{L:99d} yield $(\bB(f))^*\in\kL(H^1(\R))$ with
   \begin{align}\label{LE6}
 \sup_{t\in[0,T]}\| (\bB(f))^*\|_{\kL(H^1(\R))}\leq C.
 \end{align}
Since for $t\in(0,T]$ and $\varphi\in H^1(\R)$  
 \begin{align*}
  \Big|\int_\R \bB(f)[\oo_3(f)]\varphi\, dx\Big|&=\Big|\int_\R  \oo_3(f)(\bB(f))^*[\varphi]\, dx\Big| \leq \|(\bB(f))^*[\varphi]\|_{H^1}\|\oo_3(f)\|_{H^{-1}},
 \end{align*}
 we deduce from  \eqref{LE5}-\eqref{LE6} that 
   \begin{align}\label{DEDE3}
 \sup_{t\in(0,T]}\| \bB(f)[\oo_3(f)]\|_{H^{-1}}\leq C.
 \end{align}
  Gathering \eqref{DEDE1}, \eqref{DEDE2}, and \eqref{DEDE3} it follows that $f\in {\rm BC}^1((0,T], H^{-1}(\R))$, and together with $f\in C([0,T], H^s(\R))$ we conclude that 
  $f\in C^\eta([0,T],H^{2}(\R)).$ This proves the uniqueness claim.

  The criterion for global existence and the remaining regularity properties follow in the same way as in particular case $\mu_-=\mu_+$ (see the proofs of \cite[Theorem~1.2-1.3]{M16x}), and the details are therefore omitted.
\end{proof}

\section{The Muskat problem without surface tension}\label{Sec4}
In this section we neglect the surface tension effects, that is we set $\sigma=0 $.   
 Since  the curvature term does no longer appear  in \eqref{P:2}, we cannot expect  to express \eqref{P} as a quasilinear evolution  equation when $\mu_-\neq\mu_+$.
As a first step we formulate the problem \eqref{P} as an evolution problem for the free boundary $f$ only, cf. \eqref{Pest2},  which appears  to be,  in the regime where $\mu_-\neq \mu_+$, of fully nonlinear type.
 \medskip

\paragraph{\bf The abstract formulation}  
We start by solving  the  equation \eqref{P:2}.
Since $\sigma=0$, \eqref{P:2} is equivalent to
\begin{equation}\label{SP}
- c_{\rho,\mu }f' =(1+a_\mu\bA(f))[\oo],
\end{equation}
where
\[
c_{\rho,\mu}:=\frac{2k\Theta}{\mu_-+\mu_+}
\]
and where $\Theta$ is defined in \eqref{TET}.
Recalling Proposition~\ref{P:I2}, we remark that the equation   \eqref{SP} has a unique solution   $\oo\in H^1(\R)$, provided that the left-hand side satisfies $f'\in H^1(\R)$ and  the argument of $\bA$   belongs to  $ H^2(\R)$.
Hence, the same regularity  is required from $f$ on both sides of \eqref{SP}. 
 This is  the main reason  why the Muskat problem without surface tension is, for $\mu_-\neq \mu_+$, a fully nonlinear evolution problem.

\begin{prop}\label{P:88a} Given $f\in H^2(\R)$,  there exists a unique solution $\oo:=\oo(f)$ to  \eqref{SP} and
\begin{align}\label{reg0}
 \oo\in C^\omega(H^2(\R), H^1(\R)).
\end{align}
Moreover,  given $f_0\in H^2(\R)$ and $\tau\in(1/2,1),$ there exists a constant $C $ such that
\begin{align}
  \|\p_f\oo(f_0)[f]\|_2&\leq C\|f\|_{H^{1+\tau}},\label{DFF1}\\[1ex]
 \|\p_f\oo(f_0)[f]\|_{H^\tau}&\leq C\|f\|_{H^{1+2\tau-\tau^2}},\label{DFF3}\\[1ex]
    \|\p_f\oo(f_0)[f]\|_{H^1 }&\leq C\|f\|_{H^2}\label{DFF2}
\end{align}
for all $f\in H^2(\R)$.
\end{prop}
\begin{proof} Since $|a_\mu|<1$, it follows from Proposition~\ref{P:I2} that 
\begin{align*}
 \oo(f):= -c_{\rho,\mu}(1+a_\mu\bA(f))^{-1}[f']\in H^1(\R)
\end{align*}
is the unique solution to \eqref{SP}.
The regularity property is a consequence of \eqref{RegA}.

Let now $f_0\in H^2(\R)$ be fixed. Using the chain rule, we find that   $\p_f\oo(f_0)[f]$ solves the equation
\begin{align}\label{FD2}
 (1+a_\mu \bA(f_0))[\p_f\oo(f_0)[f]]=-c_{\rho,\mu}f'-a_\mu\p\bA(f_0)[f][\oo(f_0)],
\end{align}
where, taking advantage of Lemma~\ref{L:99a}, we find that
\begin{align}
  \pi\p\bA(f_0)[f][\oo]&=f'B_{0,1}(f_0)[\oo]-2f_0'B_{2,2}(f_0,f_0)[f_0,f,\oo]-B_{1,1}(f_0)[f,\oo]\nonumber\\[1ex]
  &\hspace{0.424cm}+2B_{3,2}(f_0,f_0)[f_0,f_0,f,\oo]\label{FD3}
\end{align}
for all $f\in H^2(\R)$ and $\oo\in H^1(\R) $.
The estimate \eqref{DFF1} is a consequence of Lemma~\ref{L:99a}~$(i)$ and \eqref{DI}, while \eqref{DFF2}  simply states that $\p_f\oo(f_0)\in\kL(H^2(\R), H^1(\R))$.
 Finally, \eqref{DFF3} follows from \eqref{DFF1}-\eqref{DFF2} via complex interpolation, cf. \eqref{CI}.  
\end{proof}

Appealing to Proposition~\ref{P:88a}, we may reformulate \eqref{P}, after rescaling the time,  as an autonomous evolution problem  
\begin{equation}\label{Pest2}
  \p_t f= \Phi(f),\quad t\geq 0,\qquad f(0)=f_0,
\end{equation}
where $\Phi:H^2(\R)\to H^1(\R)$ is the fully nonlinear and nonlocal operator
\begin{align*}
 \Phi(f):=\bB(f)[\oo(f)],\qquad f\in H^2(\R).
\end{align*}
The regularity properties  \eqref{regB} and \eqref{reg0} ensure that  
\begin{align}\label{reg22}
  \Phi\in C^{\omega}  (H^2(\R), H^1(\R)).
\end{align} 

In the analysis of \eqref{Pest2} we have to differentiate between the cases  $\Theta=0$ and $\Theta\neq 0$.
The case when  $\Theta=0$ is special, because  for this choice  $c_{\rho,\mu}=0$ and  therewith $\oo(f)=0$ for all $f\in H^r(\R)$ with $r>3/2$, cf. Theorem~\ref{T:I1}.
Hence,  the problem \eqref{P} possesses for each $f\in H^r(\R)$ with $r>3/2$ a unique global solution $f(t)=f_0$ for all $t\in\R$.
In the remaining part of the paper we  address the nondegenerate case when 
$\Theta\neq0$.
The next task it to determine the Fr\'echet derivative $\p\Phi(f_0)$,   $f_0\in H^2(\R),$   and to investigate whether  this derivative is the generator of a strongly continuous and analytic semigroup in $\kL(H^1(\R))$.
Our analysis below shows that   the operator $\p\Phi(f_0)$ has the desired generator property, provided that   $f_0$ is chosen such that the Rayleigh-Taylor condition holds.\medskip

\paragraph{\bf The Rayleigh-Taylor condition} 
Given $f_0\in H^2(\R)$,  the  Rayleigh-Taylor condition may be reexpressed, in view of \eqref{RT} and of the formulas \eqref{eq:S1} and \eqref{V2}, as
  \begin{equation}\label{RTT2}
   a_{\text{\tiny RT}}:=c_{\rho,\mu} +\frac{a_\mu}{\pi }B_{0,1}(f_0)[\oo_0] +\frac{a_\mu}{\pi }f_0'B_{1,1}(f_0)[f_0,\oo_0]>0 ,
  \end{equation}
where, to keep the notation short, we have set
\begin{align}\label{000} 
 \oo_0:=\oo(f_0)\in H^1(\R),
\end{align}
 cf. Proposition~\ref{P:88a}. If the fluids have equal viscosities, then $a_\mu=0$ and the Rayleigh-Taylor condition is equivalent to the relation $\Theta=g(\rho_--\rho_+)>0.$ 
Since $\Theta\neq0$ and  $\bB(f_0)[\oo_0]\in H^1(\R)$ vanishes at infinity,    the Rayleigh-Taylor condition  implies also for $\mu_-\neq\mu_+$ that    $\Theta>0$, and  due to this fact we restriction 
our analysis to this case. 
Finally, we remark that  \eqref{RTT2} is equivalent to
\begin{equation}\label{RTT3}
  \inf_\R a_{\text{\tiny RT}}>0,
  \end{equation}
  and therefore the set $\cO$ of  initial data that satisfy the Rayleigh-Taylor condition and that was  introduced in Section \ref{Sec0} can be reexpressed as
\[
 \cO=\Big\{f_0\in H^2(\R)\,:\,\inf_\R a_{\text{\tiny RT}}>0 \Big\}.
 \]
 Since $\oo(0)=0$, it follows that $ 0\in\cO$ and therewith all $f_0\in H^2(\R)$ that are sufficiently small belong to this set. 
We emphasize that this set may  actually be very large, for example it is easy to infer from \eqref{RTT2} that $\cO=H^2(\R)$ if $\mu_-=\mu_+.$
\medskip

\paragraph{\bf The Fr\'echet derivative} 
Let   $f_0\in \cO$. 
Keeping \eqref{000} in mind, we compute that    
\begin{align}\label{FD}
 \p\Phi(f_0)[f]=\p\bB(f_0)[f][\oo_0]+\bB(f_0)[\p_f\oo(f_0)[f]]\qquad\text{for $f\in H^2(\R)$,}
\end{align}
where
\begin{align} 
 \p\bB(f_0)[f][\oo_0]&=-2B_{2,2}(f_0,f_0)[f_0,f,\oo_0]+f'B_{1,1}(f_0)[f_0,\oo_0]+f'_0B_{1,1}(f_0)[f,\oo_0]\nonumber\\[1ex]
 &\hspace{0.424cm}-2f_0'B_{3,2}(f_0,f_0)[f,f_0,f_0,\oo_0].\label{FD1}
\end{align}
While \eqref{FD}  follows from the chain rule, the relation  \eqref{FD1}  can be easily derived with the help of Lemma~\ref{L:99a}.
In order to establish the generator property for $\p\Phi(f_0)$, we proceed in the same way as in Section \ref{Sec3}, but now the situation is much more involved. 
To begin, we consider a continuous path
\begin{align*}
 [\tau\mapsto \Psi(\tau)]:[0,1]\to\kL(H^2(\R),H^1(\R)),
\end{align*}
with
\begin{align*}
 \Psi(\tau)[f]:=\tau\p\bB(f_0)[f][\oo_0]+\bB(\tau f_0)[w(\tau)[f]] \qquad\text{for $f\in H^2(\R)$},
\end{align*}
and where $w:[0,1]\to\kL(H^2(\R),H^1(\R))$ is a further continuous path
\begin{align}
 w(\tau)[f]:=&(1+a_\mu \bA(\tau f_0))^{-1}\Big[ -c_{\rho,\mu}f'-\frac{a_\mu}{\pi}f'B_{0,1}(f_0)[\oo_0]  -\frac{a_\mu}{\pi}(1-\tau)f'f_0'B_{1,1}(f_0)[f_0,\oo_0]\nonumber\\[1ex]
 &\hspace{3.08cm}+2\frac{\tau a_\mu}{\pi}f_0'B_{2,2}(f_0,f_0)[f_0,f,\oo_0]+\frac{\tau a_\mu}{\pi}B_{1,1}(f_0)[f,\oo_0]\nonumber\\[1ex]
 &\hspace{3.08cm}-2\frac{\tau a_\mu}{\pi}B_{3,2}(f_0,f_0)[f_0, f_0,f,\oo_0]  \Big].\label{wee}
\end{align}
 The relations \eqref{FD2}-\eqref{FD3} and \eqref{FD} show that  $\Psi(1)=\p\Phi(f_0),$ while, in view of $\bA(0)=0$ and of  \eqref{RTT2}, it holds that
 \begin{align}\label{DC}
  \Psi(0)[f]=\bB(0)[w(0)[f]]=- \bB(0)[ a_{\text{\tiny RT}}f']=- \pi H[ a_{\text{\tiny RT}}f'],
 \end{align}
 with $H$ denoting   the Hilbert transform again.
 The term on the right-hand side of \eqref{wee} which has $(1-\tau )$ as a multiplying factor has been introduced artificially. 
Due to this trick, we are able for example to write, when setting $\tau=0$, the  function $a_{\text{\tiny RT}}$ as a multiplicative term in the argument of $\bB(0)$ in \eqref{DC}.
Moreover, this artificial term  provides some useful cancellations in the proof of Theorem \ref{T2} which are, together  with our assumption  \eqref{RTT3}, 
an important ingredient  when establishing the generator property for $\Psi(1)$, see  Lemma~\ref{L:Isom} and Theorem~\ref{TK3}.
 
 We note    that the operator $w$ defined in \eqref{wee} can be estimated  in a  similar manner as the Fr\'echet derivative $\p_f\oo(f_0)$, that is
 there exists a constant $C $ such that
\begin{align}
 \|w(\tau)[f]\|_2&\leq C\|f\|_{H^{7/4}},\label{DFF1*}\\[1ex]
 \|w(\tau)[f]\|_{H^{3/4}}&\leq C\|f\|_{H^{31/16}},\label{DFF3*}\\[1ex]
   \|w(\tau)[f]\|_{H^1 }&\leq C\|f\|_{H^2}\label{DFF2*}
\end{align}
for all $\tau\in [0,1]$ and all $f\in H^2(\R)$.
Additionally, recalling \eqref{DERR} and \eqref{0lote}, the Lemmas~\ref{L:99a}-\ref{L:99d}, in particular the estimate \eqref{REF2},  lead us to the following relation 
\begin{align}
 (1+a_\mu \bA(\tau f_0))[( w(\tau)[f])']
 &=-c_{\rho,\mu}f''-\frac{a_\mu}{\pi}f''B_{0,1}(f_0)[\oo_0]-\frac{a_\mu}{\pi}(1-\tau)f''f_0'B_{1,1}(f_0)[f_0,\oo_0]\nonumber\\[1ex]
 &\hspace{0.424cm}+2\frac{\tau a_\mu}{\pi}f_0'B_{1,2}(f_0,f_0)[f_0,f''\oo_0]+\frac{\tau a_\mu}{\pi}B_{0,1}(f_0)[f''\oo_0]\nonumber\\[1ex]
 &\hspace{0.424cm}-2\frac{\tau a_\mu}{\pi}B_{2,2}(f_0,f_0)[f_0, f_0,f''\oo_0]+T_{2,{\rm lot}}(\tau)[f],\label{oto0-}
\end{align}
where   the lower order terms are encompassed by the term $T_{2,{\rm lot}}(\tau)[f]$ and
\begin{align}
   \|T_{2,{\rm lot}}(\tau)[f]\|_2\leq C\|f\|_{H^{31/16}} \label{oto0}
 \end{align}
 for all $\tau\in [0,1]$ and all $f\in H^2(\R)$.
 
 The following theorem lies at the core of our generator result in Theorem~\ref{TK3}, and its assertion is independent of whether  \eqref{RTT3} holds or not. 
 Before stating the result, we point out that $B_{0,1}(0)\circ\p_x=\pi(-\p_x^2)^{1/2}, $ where   $(-\p_x^2)^{1/2}$ is the Fourier multiplier with symbol $[\xi\mapsto |\xi|].$

 \begin{thm}\label{T2} 
Let $f\in H^2(\R)$ and  $\mu>0$ be given. 
Then, there exist $\e\in(0,1)$, a finite $\e$-locali\-za\-tion family  $\{\pi_j^\e\,:\, -N+1\leq j\leq N\} $, a   constant $K=K(\e)$, 
and for each  $ j\in\{-N+1,\ldots,N\}$ and $\tau\in[0,1]$ there 
exist bounded operators $$\bA_{ j,\tau}\in\kL(H^2(\R), H^1(\R))$$
 such that 
 \begin{equation}\label{D1}
  \|\pi_j^\e \Psi(\tau) [f]-\bA_{j,\tau}[\pi^\e_j f]\|_{H^1}\leq \mu \|\pi_j^\e f\|_{H^2}+K\|  f\|_{H^{31/16}}
 \end{equation}
 for all $ j\in\{-N+1,\ldots,N\}$, $\tau\in[0,1],$ and  $f\in H^2(\R)$. The operators $\bA_{j,\tau}$ are defined  by 
  \begin{align*} 
 \bA_{j,\tau }:=- \alpha_\tau(x_j^\e) (-\p_x^2 )^{1/2}+\beta_\tau (x_j^\e)\p_x, \qquad |j|\leq N-1, 
 \end{align*}
 where  $x_j^\e\in \supp  \pi_j^\e$, 
 \begin{align*}
 \alpha_\tau:=\pi{\Big(1-\frac{\tau f_0'^2}{1+f_0'^2}\Big)} a_{\text{\tiny{\em RT}}}, \qquad  \beta_\tau:=\tau\Big(B_{1,1}(f_0)[f_0,\oo_0]-a_\mu\pi\frac{\oo_0}{1+f_0'^2}\Big),   
 \end{align*}
 and $\oo_0:=\oo(f_0)$, respectively
 \begin{align*} 
 \bA_{N,\tau }:= -  \pi c_{\rho,\mu} (-\p_x^2)^{1/2}. 
 \end{align*}
\end{thm}
\begin{proof} Let $\{\pi_j^\e\,:\, -N+1\leq j\leq N\} $ be a finite $\e$-localization family   and $\{\chi_j^\e\,:\, -N+1\leq j\leq N\} $ be an associated family, 
with $\e\in(0,1)$ to be fixed later on. As before, we denote by $C$ constants which are
independent of $\e$ (and, of course, of $f\in H^2(\R)$, $\tau\in [0,1]$, and $j \in \{-N+1, \ldots, N\}$), and the constants  denoted by $K$ may depend only upon $\e.$ \medskip
 
\noindent{\em Step 1: The lower order terms.}  
Recalling Lemma~\ref{L:99a},  \eqref{DI}, \eqref{FD1},  \eqref{DFF1*}, and exploiting the embedding $H^{7/4}(\R)\hookrightarrow W^1_\infty(\R)$ we find that
 \begin{align}
  \|\pi_j^\e \Psi(\tau) [f]-\bA_{j,\tau}[\pi^\e_j f]\|_{H^1}&\leq \|\pi_j^\e \Psi(\tau) [f]-\bA_{j,\tau}[\pi^\e_j f]\|_{2} +\|(\pi_j^\e \Psi(\tau) [f]-\bA_{j,\tau}[\pi^\e_j f])'\|_{2}\nonumber\\[1ex]
  &\leq K\|f\|_{H^{7/4}} +\|\pi_j^\e (\Psi(\tau) [f])'-\bA_{j,\tau}[(\pi^\e_j f)']\|_{2}.\label{oto1}
 \end{align}
Moreover, differentiating the relations \eqref{FD1} and \eqref{wee} once, it follows from \eqref{derB}, \eqref{DERR}, \eqref{DFF3*} and the Lemmas~\ref{L:99a}-\ref{L:99d} that  
\begin{align*}
 (\Psi(\tau) [f])'&=-2\tau  B_{1,2}( f_0, f_0)[ f_0,f''\oo_0]+\tau  f''B_{1,1}( f_0)[f_0,\oo_0] +\tau  f_0' B_{0,1}( f_0 )[ f''\oo_0]\\[1ex] 
 &\hspace{0.424cm}-2\tau f_0' B_{2,2}( f_0,f_0)[ f_0,f_0,f''\oo_0]  + \bB(\tau f_0)[(w(\tau)[f])'] +T_{3,{\rm lot}}(\tau)[f],
\end{align*}
 with 
\begin{align}
   \|T_{3,{\rm lot}}(\tau)[f]\|_2\leq C(\|f\|_{H^{7/4}}+\|w(\tau)[f]\|_{H^{3/4}})\leq C\|f\|_{H^{31/16}}.  \label{oto2}
 \end{align}
 Hence, we are left to estimate the $L_2$-norm of the leading order term
 \begin{align*}
&-2\tau  B_{1,2}( f_0, f_0)[ f_0,f''\oo_0]+\tau  f''B_{1,1}( f_0)[f_0,\oo_0] +\tau  f_0' B_{0,1}( f_0 )[ f''\oo_0]\\[1ex] 
 &\hspace{0.924cm}-2\tau f_0' B_{2,2}( f_0,f_0)[ f_0,f_0,f''\oo_0]  + \bB(\tau f_0)[(w(\tau)[f])'] -\bA_{j,\tau}[(\pi^\e_j f)',
\end{align*}
 and this is performed below in several steps.
 \medskip
 
\noindent {\em Step 2.}
Given   $|j|\leq N-1$, we  let $ \bA_{j}^{k}\in\kL(H^2(\R),H^1(\R)$), $1\leq k\leq 4,$ be the operators defined by
\begin{equation}\label{Aj}
\begin{aligned}
 & \bA_{j}^1:=\frac{ f_0'(x_j^\e)\oo_0(x_j^\e)}{(1+f_0'^2(x_j^\e))^2}B_{0,1}(0)\circ\p_x,\qquad \bA_{j}^2:=B_{1,1}( f_0)[ f_0,\oo_0](x_j^\e)\p_x,\\[1ex]
 & \bA_{j}^3:=\frac{ f_0'(x_j^\e)\oo_0(x_j^\e)}{ 1+f_0'^2(x_j^\e) }B_{0,1}(0)\circ\p_x,\qquad \bA_{j}^4:=\frac{ f_0'^3(x_j^\e)\oo_0(x_j^\e)}{(1+f_0'^2(x_j^\e))^2}B_{0,1}(0)\circ\p_x,
\end{aligned}
\end{equation}
with $x_j^\e\in \supp \pi_j^\e.$ 
We   prove in this step  that 
\begin{align}
 &2\|\pi_j^\e B_{1,2}( f_0, f_0)[ f_0,f''\oo_0]-\bA_{j}^1[(\pi^\e_j f)']\|_{2}+\|\pi_j^\e f''B_{1,1}( f_0)[f_0,\oo_0]-\bA_{j}^2[(\pi^\e_j f)']\|_{2}\nonumber\\[1ex]
 &+\|\pi_j^\e f_0'B_{0,1}(  f_0)[ f''\oo_0]-\bA_{j}^3[(\pi^\e_j f)']\|_{2}+2\|\pi_j^\e f_0' B_{2,2}( f_0, f_0)[  f_0, f_0,f''\oo_0]-\bA_{j}^4[(\pi^\e_j f)']\|_{2}\nonumber\\[1ex]
 &\hspace{1cm}\leq \frac{\mu}{2} \|\pi_j^\e f\|_{H^2}+K\|f\|_{H^{7/4}}\label{oto3}
\end{align}
for all   $ |j|\leq N-1$ and  $f\in H^2(\R)$, provided that $\e\in(0,1)$ is sufficiently  small.

Since $B_{1,1}( f_0)[ f_0,\oo_0]\in H^1(\R)\subset {\rm BC}^{1/2}(\R)$, for $\e$ sufficiently small  we have 
\begin{align}
 &\hspace{-1cm}\|\pi_j^\e f''B_{1,1}( f_0)[f_0,\oo_0]-\bA_{j}^2[(\pi^\e_j f)']\|_{2}\nonumber\\[1ex]
 &=\|\pi_j^\e f''B_{1,1}( f_0)[f_0,\oo_0]- B_{1,1}( f_0)[ f_0,\oo_0](x_j^\e)(\pi^\e_j f)''\|_2\nonumber\\[1ex]
 &\leq   \|(\pi_j^\e f)''\|_2\|\chi_j^\e\big(B_{1,1}(f_0)[ f_0,\oo_0]- B_{1,1}(f_0)[f_0,\oo_0](x_j^\e)\big)\|_\infty+K\|  f\|_{H^{1}}\nonumber\\[1ex] 
 &\leq \frac{\mu}{8} \|\pi_j^\e f\|_{H^2}+K\|  f\|_{H^{1}}. \label{ca01}
\end{align}

The arguments used to estimate the remaining three terms in \eqref{oto3}  are similar, and therefore we only present in detail those for the first term.
To begin, we write
 \begin{align*}
  \pi_j^\e B_{1,2}( f_0, f_0)[f_0,f''\oo_0]-\bA_{j}^1[\pi_j^\e f]=T_{1}[f]+T_{2}[f]+T_{3}[f],
 \end{align*}
where
\begin{align*}
 T_{1}[f]&:= \pi_j^\e B_{1,2}( f_0, f_0)[ f_0,f''\oo_0]-  B_{1,2}(f_0, f_0)[ f_0,\pi_j^\e f''\oo_0],\\[1ex]
 T_{2}[f]&:=    B_{1,2}( f_0, f_0)[f_0,\pi_j^\e f''\oo_0]- \frac{f_0'(x_j^\e)}{(1+f_0'^2(x_j^\e))^2}B_{0,1}(0)[\pi_j^\e f''\oo_0],\\[1ex]
 T_{3}[f]&:=      \frac{ f_0'(x_j^\e)}{(1+f_0'^2(x_j^\e))^2}\big(B_{0,1}(0)[\pi_j^\e f''\oo_0]-\oo_0(x_j^\e)B_{0,1}(0)[(\pi_j^\e f)'']\big).
\end{align*}
The term $T_{1}[f] $ can be estimate in the same way as  the term $T_{11}[h]$  in the proof of Theorem~\ref{TK1}. Indeed, integration by parts  together with Lemma~\ref{L:99a} yields
\begin{align}\label{ca11}
 \|T_{1}[f]\|_2\leq K\|f\|_{H^{7/4}} 
\end{align}
 for all $-N+1\leq j\leq N$.
 
 Concerning $T_{2}[f],$ we appeal to  \eqref{rell} and   write
 \begin{align*}
  T_{2}[f]&=B_{1,2}( f_0, f_0)[ f_0,\pi_j^\e f''\oo_0]- B_{1,2}( f_0'(x_j^\e){\rm id}_{\R}, f_0'(x_j^\e){\rm id}_{\R})[ f_0'(x_j^\e){\rm id}_{\R},\pi_j^\e f''\oo_0]\\[1ex]
  &= B_{1,2}( f_0, f_0)[f_0- f_0'(x_j^\e){\rm id}_{\R},\pi_j^\e f''\oo_0]\\[1ex]
  &\hspace{0.424cm}-B_{3,3}( f_0, f_0,   f_0'(x_j^\e){\rm id}_{\R}) [f_0'(x_j^\e){\rm id}_{\R}, f_0- f_0'(x_j^\e){\rm id}_{\R},f_0+ f_0'(x_j^\e){\rm id}_{\R},\pi_j^\e f''\oo_0]\\[1ex]
  &\hspace{0.424cm} - B_{3,3}( f_0,   f_0'(x_j^\e){\rm id}_{\R},   f_0'(x_j^\e){\rm id}_{\R}) [f_0'(x_j^\e){\rm id}_{\R},f_0- f_0'(x_j^\e){\rm id}_{\R},f_0+ f_0'(x_j^\e){\rm id}_{\R}, \pi_j^\e f''\oo_0],
 \end{align*}
and, since $\chi_j^\e=1$ on $\supp \pi_j^\e$, we further have $T_{2}[f]=T_{21}[f]-T_{22}[f],$ where
  \begin{align*}
  T_{21}[f] &:= \chi_j^\e B_{1,2}( f_0, f_0)[f_0- f_0'(x_j^\e){\rm id}_{\R},\pi_j^\e f''\oo_0]\\[1ex]
  & \hspace{0.424cm}\phantom{:}-\chi_j^\e B_{3,3}( f_0, f_0,   f_0'(x_j^\e){\rm id}_{\R}) [f_0'(x_j^\e){\rm id}_{\R},f_0- f_0'(x_j^\e){\rm id}_{\R},f_0+ f_0'(x_j^\e){\rm id}_{\R},\pi_j^\e f''\oo_0]\\[1ex]
  &\hspace{0.424cm} \phantom{:}- \chi_j^\e B_{3,3}( f_0, f_0'(x_j^\e){\rm id}_{\R},   f_0'(x_j^\e){\rm id}_{\R}) [f_0'(x_j^\e){\rm id}_{\R}, f_0- f_0'(x_j^\e){\rm id}_{\R},f_0+ f_0'(x_j^\e){\rm id}_{\R},\pi_j^\e f''\oo_0],
 \end{align*}
 and 
  \begin{align*}
  T_{22}[f]&:=\PV\int_\R\frac{\big[\delta_{[\cdot,y]}(f_0- f_0'(x_j^\e){\rm id}_{\R})/y\big]\big[\delta_{[\cdot,y]}\chi_j^\e/y \big]}{\big[1+\big(\delta_{[\cdot,y]}f_0/y\big)^2\big]^2} \tau_y (\pi_j^\e f''\oo_0) \, dy\\[1ex]
  &\hspace{0.424cm} \phantom{:}-f_0'(x_j^\e)\PV\int_\R\frac{\big[\delta_{[\cdot,y]}(f_0- f_0'(x_j^\e){\rm id}_{\R})/y\big]\!\big[\delta_{[\cdot,y]}(f_0+f_0'(x_j^\e){\rm id}_{\R})/y\big]\!
  \big[\delta_{[\cdot,y]}\chi_j^\e/y \big]}{\big[1+\big(\delta_{[\cdot,y]}(f_0'(x_j^\e){\rm id}_{\R})/y\big)^2\big]
  \big[1+\big(\delta_{[\cdot,y]}f_0/y\big)^2\big]^2}  \tau_y (\pi_j^\e f''\oo_0) \, dy\\[1ex]
  &\hspace{0.424cm} \phantom{:}-f_0'(x_j^\e)\PV\int_\R\frac{\big[\delta_{[\cdot,y]}(f_0- f_0'(x_j^\e){\rm id}_{\R})/y\big]\!\big[\delta_{[\cdot,y]}(f_0+f_0'(x_j^\e){\rm id}_{\R})/y\big]\!
  \big[\delta_{[\cdot,y]}\chi_j^\e/y \big]}{\big[1+\big(\delta_{[\cdot,y]}(f_0'(x_j^\e){\rm id}_{\R})/y\big)^2\big]^2
  \big[1+\big(\delta_{[\cdot,y]}f_0/y\big)^2\big]} \tau_y (\pi_j^\e f''\oo_0) \, dy.
 \end{align*}
 In this formula we denote by $\{\tau_y\}_{y\in\R}$ the   translation $C_0$-group on $L_2(\R)$ introduced in Lemma \ref{L:99e}.  
 Integrating by parts, we get
 \begin{align*}
 \|T_{22}[f]\|_2\leq K\|f\|_{H^{7/4}}.
\end{align*}
Furthermore,   the arguments used to derive  estimate \eqref{LL2aaa} lead us to 
 \begin{align*} 
 \|T_{21}[f]\|_2\leq C\|f'_0-f'_0(x_j^\e)\|_{L_\infty(\supp\chi_j^\e)}\|\pi_j^\e f''\oo_0\|_2\leq \frac{\mu}{16}  \|\pi_j^\e f\|_{H^2}+K\|  f\|_{H^{1}},
\end{align*} 
provided that $\e$ is sufficiently small.
Hence,  for $\e$ sufficiently small and $|j|\leq N-1,$ it holds that
 \begin{align}\label{ca12}
 \|T_{2}[f]\|_2\leq   \frac{\mu}{32}  \|\pi_j^\e f\|_{H^2}+K\|  f\|_{H^{7/4}}.
\end{align}
Finally, since $\oo_0\in H^1(\R)\subset {\rm BC}^{1/2}(\R)$  we get 
\begin{align*}
 \|T_3[f]\|_2&\leq \|B_{0,1}(0)[(\pi_j^\e f)''(\oo_0-\oo_0(x_j^\e))]\|_2+K\|f\|_{H^1}\nonumber\\[1ex]
 &\leq C\|(\pi_j^\e f)''(\oo_0-\oo_0(x_j^\e))\|_2+K\|f\|_{H^1}\nonumber\\[1ex]
  &\leq C \|\chi_j^\e(\oo_0-\oo_0(x_j^\e))\|_\infty\|\pi_j^\e f\|_{H^2}+K\|f\|_{H^1}\nonumber\\[1ex]
  &\leq  \frac{\mu}{32}\|\pi_j^\e f\|_{H^2}+K\|f\|_{H^1},
\end{align*}
provided that  $\e$ is sufficiently small, and together with  \eqref{ca11} and  \eqref{ca12} we obtain 
\begin{align*}
 2\|\pi_j^\e B_{1,2}( f_0, f_0)[ f_0,f''\oo_0]-\bA_{j}^1[(\pi^\e_j f)']\|_{2}\leq \frac{\mu}{8} \|\pi_j^\e f\|_{H^2}+K\|f\|_{H^{7/4}}.
\end{align*}
Since the other two terms in \eqref{oto3} can be estimated in the same way, we conclude, in view of \eqref{ca01}, that  \eqref{oto3} is satisfied. 
\medskip

 \noindent{\em Step 3.} Given $\tau\in[0,1]$, we let
\begin{align*} 
 c_\tau:=a_{\text{\tiny RT}}-\frac{\tau a_\mu}{\pi}f_0'B_{1,1}(f_0)[f_0,\oo_0]\in H^1(\R).
\end{align*}
Moreover, for each  $l\in\{0,1\}$, $\tau\in[0,1],$ and $|j|\leq N-1,$ we let  $ \bA_{j,\tau}^{5+l}\in\kL(H^2(\R), H^1(\R))$ denote  the operator  
\begin{align*} 
 \bA_{j,\tau}^{5+l}:=\frac{ f_{l,\tau}'^2(x_j^\e)}{1+\tau^2f_0'^2(x_j^\e)}\Big[ - c_\tau(x_j^\e) B_{0,1}(0)\circ\p_x- \tau a_\mu\pi\frac{\oo_0(x_j^\e)}{ 1+ f_0'^2(x_j^\e) } \p_x\Big],
\end{align*}
where $x_j^\e\in \supp \pi_j^\e$ and $f_{l,\tau} :=(1-l)\tau f_0+l\, {\rm id}_\R. $
 We next prove  that
\begin{align}\label{T4T}
 \Big\|\pi_j^\e \bB(\tau f_0)[(w(\tau)[f])']-\sum_{l=0}^1\bA_{j,\tau}^{5+l}[(\pi^\e_j f)']\Big\|_{2}\leq \frac{\mu}{2} \|\pi_j^\e f\|_{H^2}+K\|  f\|_{H^{31/16}}
\end{align}
 for all $ |j|\leq N-1$, $\tau\in[0,1],$ and  $f\in H^2(\R)$, provided that $\e$ is chosen sufficiently small.
 We begin by observing that
 \[
  \bB(\tau f_0)=\sum_{l=0}^1f_{l,\tau}'B_{1,1}(\tau f_0)[f_{l,\tau},\,\cdot\,],
 \]
and therefore
 \begin{align*}
  \pi_j^\e \bB(\tau f_0)[(w(\tau)[f])']-\sum_{l=0}^1\bA_{j,\tau}^{5+l}[(\pi^\e_j f)']=\sum_{l=0}^1\Big(\pi_j^\e f_{l,\tau}'B_{1,1}(\tau f_0)[f_{l,\tau}, (w(\tau)[f])']-\bA_{j,\tau}^{5+l}[(\pi^\e_j f)']\Big).
 \end{align*}
We now decompose
\[\pi_j^\e f_{l,\tau}'B_{1,1}(\tau f_0)[f_{l,\tau}, (w(\tau)[f])']-\bA_{j,\tau}^{5+l}[(\pi^\e_j f)']=T_{4}[f]+T_{5}[f]+T_{6}[f],\]
where
\begin{align*}
 T_{4}[f]&:=   \pi_j^\e f_{l,\tau}'B_{1,1}(\tau f_0)[f_{l,\tau}, (w(\tau)[f])']-   f_{l,\tau}' (x_j^\e)B_{1,1}(\tau f_0)[f_{l,\tau},\pi_j^\e (w(\tau)[f])'],\\[1ex]
 T_{5}[f]&:=  f_{l,\tau}' (x_j^\e)B_{1,1}(\tau f_0)[f_{l,\tau},\pi_j^\e (w(\tau)[f])']-  \frac{f_{l,\tau}'^2(x_j^\e)}{1+\tau^2f_0'^2(x_j^\e)}  B_{0,1}(0)[ \pi_j^\e (w(\tau)[f])'],\\[1ex]
 T_{6}[f]&:=  \frac{f_{l,\tau}'^2(x_j^\e)}{1+\tau^2f_0'^2(x_j^\e)}  B_{0,1}(0)[ \pi_j^\e (w(\tau)[f])']- \bA_{j,\tau}^{5+l}[(\pi_j^\e f)'].
\end{align*}
The arguments used to derive  \eqref{LL2c}  combined with \eqref{DFF1*} imply that
\begin{align}\label{ca41a}
 \|T_{4}[f]\|_2\leq C\|\chi_j^\e(f_{l,\tau}'-f_{l,\tau}'(x_j^\e))  \|_\infty\|\pi^\e_j(w(\tau)[f])' \|_2 +K\|f\|_{H^{7/4}},
\end{align}
  and we are left to estimate $\|\pi^\e_j(w(\tau)[f])' \|_2.$
To this end we compute for $j\in\{-N+1,\ldots, N\}$ that 
\begin{align}
 (1+a_\mu\bA(\tau f_0))[\pi_j^\e(w(\tau)[f])' ]=\pi_j^\e(1+a_\mu\bA(\tau f_0))[(w(\tau)[f])']-\frac{  a_\mu}{\pi}T_{1j, {\rm lot}}(\tau f_0)[w(\tau)[f]],\label{ERG1}
\end{align}
where $T_{1j, {\rm lot}}$ are defined in \eqref{T1j}.
Moreover, recalling \eqref{oto0-}, it follows  that 
\begin{align}
 \pi_j^\e(1+a_\mu\bA(\tau f_0))[(w(\tau)[f])'] &=-c_\tau \pi_j^\e f''+2\frac{\tau a_\mu}{\pi}f_0'B_{1,2}(f_0,f_0)[f_0,\pi_j^\e f''\oo_0]\nonumber\\[1ex]
 &\hspace{0.424cm}+\frac{\tau a_\mu}{\pi}B_{0,1}(f_0)[\pi_j^\e f''\oo_0]-2\frac{\tau a_\mu}{\pi}B_{2,2}(f_0,f_0)[f_0, f_0,\pi_j^\e f''\oo_0]\nonumber\\[1ex]
 &\hspace{0.424cm}+\pi_j^\e T_{2,{\rm lot}}(\tau)[f] +T_{2j, {\rm lot}}(\tau)[f],\label{ERG3}
\end{align}
for $j\in\{-N+1,\ldots,N\} $, where
\begin{align*}
T_{2j, {\rm lot}}(\tau)[f]&:=2\frac{\tau a_\mu}{\pi}f_0'\big(\pi_j^\e B_{1,2}(f_0,f_0)[f_0, f''\oo_0]-B_{1,2}(f_0,f_0)[f_0,\pi_j^\e f''\oo_0]\big)\\[1ex]
 &\hspace{0.424cm}\phantom{:}+\frac{\tau a_\mu}{\pi}\big(\pi_j^\e B_{0,1}(f_0)[ f''\oo_0]-B_{0,1}(f_0)[\pi_j^\e f''\oo_0]\big)\\[1ex]
 &\hspace{0.424cm}\phantom{:}-2\frac{\tau a_\mu}{\pi}\big(\pi_j^\e B_{2,2}(f_0,f_0)[f_0, f_0, f''\oo_0]-B_{2,2}(f_0,f_0)[f_0, f_0,\pi_j^\e f''\oo_0]\big).
\end{align*}
Integrating by parts,   Lemma~\ref{L:99a}  leads us to
\begin{align}\label{ERG4}
 \|T_{2j, {\rm lot}}(\tau)[f]\|_2\leq K\|f\|_{H^{7/4}}.
\end{align}
In view of \eqref{DI}, we deduce from  \eqref{DFF1*}, \eqref{oto0}, and \eqref{ERG1}-\eqref{ERG4} that 
\begin{align}\label{Ij}
 \max_{-N+1\leq j\leq N}\|\pi_j^\e (w(\tau)[f])'\|_2\leq C\|\pi_j^\e f\|_{H^2}+K\|f\|_{H^{31/16}},
\end{align}
and   \eqref{ca41a}  yields, for $\e$  sufficiently small, that 
\begin{align}\label{ca41}
 \|T_{4}[f]\|_2\leq \frac{\mu}{6}\|\pi_j^\e f\|_{H^2} +K\|f\|_{H^{31/16}}.
\end{align}

Invoking \eqref{Ij},  the arguments used to deduce \eqref{LL2aaa} show that 
\begin{align}\label{ca42}
 \|T_{5}[f]\|_2\leq \frac{\mu}{6}\|\pi_j^\e f\|_{H^2} +K\|f\|_{H^{31/16}}
\end{align}
for all $|j|\leq N-1$, provided that $\e$ is  sufficiently small.

Concerning the last term $T_{6}[f]$, we first recall that $H^2=-{\rm id}_{L_2(\R)}$, cf. \cite{St93},  and therewith we get
$$\big(B_{0,1}(0)\big)^2=\pi^2 H^2=-\pi^2{\rm id}_{L_2(\R)}.$$

Using the latter relation together with the definition of $\bA_{j,\tau}^{5+l}$, it follows that
\begin{align}
 \|T_6[f]\|_2&\leq C\Big\|\pi_j^\e(w(\tau)[f])'+c_\tau(x_j^\e)(\pi_j^\e f)''-\frac{\tau a_\mu}{\pi}\frac{\oo_0(x_j^\e)}{ 1+ f_0'^2(x_j^\e) }B_{0,1}(0)[(\pi_j^\e f)'']\Big\|_2\nonumber\\[1ex]
 &\leq C\Big\|\pi_j^\e(w(\tau)[f])'+c_\tau(x_j^\e)\pi_j^\e f''-\frac{\tau a_\mu}{\pi}\frac{\oo_0(x_j^\e)}{ 1+ f_0'^2(x_j^\e) }B_{0,1}(0)[\pi_j^\e f'']\Big\|_2+K\|f\|_{H^1}.\label{GFG1}
\end{align}
Furthermore,   \eqref{FOF} and \eqref{ERG3} lead us to the following identity
\begin{align}
 &\hspace{-1cm}\pi_j^\e(w(\tau)[f])'+c_\tau(x_j^\e)\pi_j^\e f''-\frac{\tau a_\mu}{\pi}\frac{\oo_0(x_j^\e)}{ 1+ f_0'^2(x_j^\e) }B_{0,1}(0)[\pi_j^\e f'']\nonumber\\[1ex]
 &= (c_\tau (x_j^\e)-c_\tau)\pi_j^\e f''+2\frac{\tau a_\mu}{\pi}\Big[f_0'B_{1,2}(f_0,f_0)[f_0,\pi_j^\e f''\oo_0] -\frac{ f_0'^2(x_j^\e)\oo_0(x_j^\e)}{(1+ f_0'^2(x_j^\e))^2}B_{0,1}[\pi_j^\e f'']\Big]\nonumber\\[1ex]
 &\hspace{0.424cm}+\frac{\tau a_\mu}{\pi}\Big[ B_{0,1}(f_0)[\pi_j^\e f''\oo_0]- \frac{\oo_0(x_j^\e)}{ 1+ f_0'^2(x_j^\e) }B_{0,1}[\pi_j^\e f'']\Big]\nonumber\\[1ex]
 &\hspace{0.424cm}-2\frac{\tau a_\mu}{\pi}\Big[  B_{2,2}(f_0,f_0)[f_0, f_0,\pi_j^\e f''\oo_0]-\frac{\tau f_0'^2(x_j^\e)\oo_0(x_j^\e)}{(1+ f_0'^2(x_j^\e))^2}B_{0,1}[\pi_j^\e f'']\Big]\nonumber\\[1ex]
 &\hspace{0.424cm}+\pi_j^\e T_{2,{\rm lot}}(\tau)[f] +T_{2j, {\rm lot}}(\tau)[f]\nonumber\\[1ex]
 &\hspace{0.424cm}-\frac{a_\mu}{\pi}\Big[\tau f_0'\pi_j^\e B_{0,1}(\tau f_0)[(w(\tau)[f])'] -\frac{ \tau f_0'(x_j^\e) }{ 1+ \tau^2 f_0'^2(x_j^\e) }B_{0,1}[\pi_j^\e (w(\tau)[f])']\Big] \nonumber\\[1ex] 
 &\hspace{0.424cm}+\frac{a_\mu}{\pi}\Big[\pi_j^\e B_{1,1}(\tau f_0)[\tau f_0,(w(\tau)[f])'] -\frac{ \tau f_0'(x_j^\e) }{ 1+ \tau^2 f_0'^2(x_j^\e) }B_{0,1}[\pi_j^\e (w(\tau)[f])']\Big]. \label{GFG2}
\end{align}
In order to  estimate the first  term   on the right-hand side of \eqref{GFG2}  we rely on the property that $c_\tau\in {\rm BC}^{1/2}(\R).$
The next three terms  are of the same type as those estimated in  {\em Step 2} above,  while each of  the last two expressions can be written as a sum of two terms for
which we can use the arguments that led to \eqref{ca41a} and \eqref{ca42}.
Altogether, we obtain 
\begin{align} 
 \Big\|\pi_j^\e(w(\tau)[f])'+c_\tau(x_j^\e)\pi_j^\e f''-\frac{\tau a_\mu}{\pi}\frac{\oo_0(x_j^\e)}{ 1+ f_0'^2(x_j^\e) }B_{0,1}(0)[\pi_j^\e f'']\Big\|_{2}\leq \frac{\mu}{6C}\|\pi_j^\e f\|_{H^2} +K\|f\|_{H^{31/16}}\label{oto4}
\end{align}
for all $|j|\leq N-1$, provided that $\e$ is sufficiently small. 
In \eqref{oto4} we denote  by $C$ the positive constant which appears in \eqref{GFG1}. 
Hence, recalling \eqref{GFG1}, we get 
\begin{align}
 \|T_6[f]\|_2\leq \frac{\mu}{6}\|\pi_j^\e f\|_{H^2} +K\|f\|_{H^{31/16}}\label{oto5}.
\end{align} 
The estimate \eqref{T4T} follows now from \eqref{ca41a}, \eqref{ca42}, and \eqref{oto5}.\medskip

 \noindent{\em Step 4.} Gathering \eqref{oto1}, \eqref{oto2},  \eqref{oto3}, and \eqref{T4T}, we conclude that the estimate \eqref{D1} holds for $|j|\leq N-1$ provided that
\begin{align}\label{LR}
 \bA_{j,\tau} =-2\tau \bA_{j}^1+\tau\bA_{j}^2+\tau \bA_{j}^3-2\tau \bA_{j }^4+\sum_{l=0}^1 \bA_{j,\tau}^{5+l}.
\end{align}
Exploiting the fact that $\oo_0$  is the solution to the   equation $\oo_0=-c_{\rho,\mu}f_0'-a_\mu\bA(f_0)[\oo_0]$,  it follows at once that \eqref{LR} is satisfied. 
This completes the proof of \eqref{D1} for $|j|\leq N-1.$\medskip

 \noindent{\em Step 5.} We are left to prove \eqref{D1} for $j=N.$ This estimate  follows   by combining arguments from the previous steps with those presented in the second part of the proof of Theorem~\ref{TK1}, 
 and therefore we omit    the lengthy details.
\end{proof}

We now reconsider  the Fourier multipliers $\bA_{\tau,j}$ found in Theorem~\ref{T2} and  we notice that if $f_0$ is chosen such that the Rayleigh-Taylor condition \eqref{RTT3} holds, then there exists a constant $\eta>0$  with the property  that 
\[\eta\leq \alpha_\tau\leq1/\eta\quad\text{and}\quad |\beta_\tau|\leq 1/\eta \qquad\text{for all $\tau\in[0,1]$.} \]
Moreover, letting $\bA_{\alpha,\beta}$, with $\alpha\in[\eta,1/\eta]$ and $|\beta|\leq 1/\eta$, denote the Fourier multiplier
\[
\bA_{\alpha,\beta}:=-\alpha (-\p_x^2)^{1/2}+\beta\p_x,
\]
we may find, similarly as in \cite[Proposition 4.3]{M16x}, a constant  $\kappa_0\geq1$  such that  
  \begin{align}
 \kappa_0\|(\lambda-\bA_{\alpha,\beta})[f]\|_{H^1}\geq  |\lambda|\cdot\|f\|_{H^1}+\|f\|_{H^2}\label{14K**}
\end{align}
for all $\alpha\in[\eta,1/\eta]$, $|\beta|\leq 1/\eta,$   $\lambda\in\C$ with $\re \lambda\geq 1$, and $f\in H^2(\R)$.

In order to establish the desired generation result for $\p\Phi(f_0)=\Psi(1)$, we next show that the operator $\omega-\Psi(0)$, with $\Psi(0)$ defined in \eqref{DC}, is invertible  for large $\omega>0$. 
In contrast to the analysis in Section \ref{Sec3}, where the invertibility of  the translation $\omega-\Phi_{\sigma,1}(0)$, with $\omega>0$, follows easily from the fact that this operator is a Fourier multiplier, cf. Theorem~\ref{TK2},
a more involved analysis is required in order to establish the invertibility of $\omega-\Psi(0)$.

\begin{lemma}\label{L:Isom}
Given $\delta>0$ and $h\in H^1(\R)$, let $a:=\delta+h $ and assume that  
\[\inf_\R a>0.\] 
Then, there exists $\omega_0>0$ such that  
\begin{align*}
 \lambda+H[ a\p_x]\in{\rm Isom }(H^2(\R),H^1(\R))
 \end{align*}
  for all $\lambda\in[\omega_0,\infty)$.
\end{lemma}
\begin{proof} 
We introduce the continuous path $[\tau\mapsto B(\tau)]:[0,1]\to \kL(H^2(\R), H^1(\R))$, where
\[
B(\tau):=  H[ a_{\tau,\delta}\p_x]\qquad \text{and}\qquad a_{\tau,\delta}:=(1-\tau)\delta+\tau a,
\]
 and we prove that there exist constants $\omega_0>0$ and  $C>0$ with the property that 
 \begin{align}\label{EDY}
  \|(\lambda+B(\tau))[f]\|_{ H^1(\R)}\geq C\|f\|_{H^2(\R)}
 \end{align}
 for all $\tau\in[0,1]$, $\lambda\in[\omega_0,\infty),$ and $f\in H^2(\R)$.
Having established \eqref{EDY},   the method of continuity  together with  $B(1)= H[ a\p_x]$  and the observation that  $\lambda+B(0)$ is an invertible Fourier multiplier (the symbol of
$\lambda+B(0)$ is $m_\lambda(\xi):= \lambda+\delta |\xi|$, $\xi \in\R$) yields the desired claim.

We first prove that, given $\mu>0$, there exists a finite $\e$-localization family  $\{\pi_j^\e\,:\, -N+1\leq j\leq N\}$, a constant $K=K(\e)$, 
and for each  $ j\in\{-N+1,\ldots,N\}$ and $\tau\in[0,1]$ there 
exist bounded operators $$\bB_{ j,\tau}\in\kL(H^2(\R), H^1(\R))$$
 such that 
 \begin{equation}\label{D1+}
  \|\pi_j^\e B(\tau) [f]-\bB_{j,\tau}[\pi^\e_j f]\|_{H^1}\leq \mu \|\pi_j^\e f\|_{H^2}+K\|  f\|_{H^{7/4}}
 \end{equation}
 for all $ j\in\{-N+1,\ldots,N\}$, $\tau\in[0,1],$ and  $f\in H^2(\R)$. 
 The operators $\bB_{j,\tau}$ are defined  by 
  \begin{align*} 
 \bB_{j,\tau }:= a_{\tau,\delta}(x_j^\e)H\circ \p_x, \qquad |j|\leq N-1, 
 \end{align*}
 with $x_j^\e\in\supp  \pi_j^\e$, respectively
\begin{align*} 
 \bB_{N,\tau }:= \delta H\circ \p_x. 
 \end{align*}
 Indeed, for $|j|\leq N-1$ we obtain by using integration by parts that
 \begin{align*}
    \|\pi_j^\e B(\tau) [f]-\bB_{j,\tau}[\pi^\e_j f]\|_{H^1}&\leq K\|f\|_{H^{7/4}}+\|\pi_j^\e H[ a_{\tau,\delta} f'']-a_{\tau,\delta}(x_j^\e)H[\pi_j^\e f'']\|_2\\[1ex]
    &\leq K\|f\|_{H^{7/4}}+\|\pi_j^\e H[ a_{\tau,\delta} f'']- H[ \pi_j^\e a_{\tau,\delta} f'']\|_2\\[1ex]
    &\hspace{0.424cm}+\|H[( a_{\tau,\delta}  -a_{\tau,\delta}(x_j^\e))\pi_j^\e f'']\|_2\\[1ex]
    &   \leq K\|f\|_{H^{7/4}}+\|\chi_j^\e( a_{\tau,\delta}  -a_{\tau,\delta}(x_j^\e))\|_\infty \|(\pi_j^\e f)''\|_2,
 \end{align*}
and \eqref{D1+} holds for $|j|\leq N-1$, if $\e$ is sufficiently small. Similarly,
\begin{align*}
    \|\pi_N^\e B(\tau) [f]-\bB_{N,\tau}[\pi^\e_N f]\|_{H^1}&\leq K\|f\|_{H^{7/4}}+\|\pi_N^\e H[ a_{\tau,\delta} f'']-\delta H[\pi_N^\e f'']\|_2\\[1ex]
    &\leq K\|f\|_{H^{7/4}}+\|H[( a_{\tau,\delta}  -\delta )\pi_N^\e f'']\|_2\\[1ex]
     &  \leq K\|f\|_{H^{7/4}}+\|\chi_N^\e h\|_\infty \|(\pi_N^\e f)''\|_2
 \end{align*}
 and, since $h$ vanishes at infinity, we see that \eqref{D1+} holds also for $j=N$, provided that  $\e$ is sufficiently small.
 
 Since   $\cup_{\tau\in[0,1]} \overline{a_{\tau,\delta}(\R)}\subset[\eta,1/\eta]$ for some $\eta>0$, we may find a constant  $\kappa_0\geq1$  such that  
  \begin{align}
\label{14K+}
\kappa_0\|(\lambda+\alpha H\circ \p_x)[f]\|_{H^1}\geq  \lambda\|f\|_{H^1}+    \|f\|_{H^2}
\end{align}
for all $\alpha\in \cup_{\tau\in[0,1]} \overline{a_{\tau,\delta}(\R)}$,    $\lambda\in [1,\infty)$, and $f\in H^2(\R)$.
Choosing $\mu=1/2\kappa_0$ in \eqref{D1+}, it follows from \eqref{14K+} that for all $j\in\{-N+1,\ldots,N\}$, $\tau\in[0,1]$, $\lambda\in[1,\infty),$ and $f\in H^2(\R)$ we have
 \begin{align*}
  \kappa_0 \|\pi_j^\e (\lambda+B(\tau)) [ f]\|_{H^1}&\geq \kappa_0\|(\lambda+\bB_{j,\tau})[\pi^\e_j f]\|_{H^1}-   \|\pi_j^\e B(\tau) [ f]-\bB_{j,\tau}[\pi^\e_j f]\|_{H^1}\\[1ex]
  &\geq \frac{1}{2 } \|\pi_j^\e f\|_{H^2}+\lambda\|\pi_j^\e f\|_{H^1}-\kappa_0K\|  f\|_{H^{7/4}}
 \end{align*}
and, summing over $j\in\{-N+1,\ldots,N\}$,  we conclude together with Young's inequality and Remark~\ref{R:3}   that there exist positive constants $\omega_0$ and $C$ such that \eqref{EDY} holds for all
$\tau\in[0,1]$, $\lambda\in[\omega_0,\infty),$ and $f\in H^2(\R)$.
\end{proof}

We are now in the position to establish the desired generator property.
To this end we recall from \eqref{RTT3} that  the set $\cO$ introduced  in Section \ref{Sec0} is actually given by
\[
\cO=\Big\{f_0\in H^2(\R)\,:\, \inf_\R \Big[c_{\rho,\mu} +\frac{a_\mu}{\pi }B_{0,1}(f_0)[\oo_0] +\frac{a_\mu}{\pi }f_0'B_{1,1}(f_0)[f_0,\oo_0]\Big]>0\Big\},
\]
and we note that the continuity of the mappings 
\[\big[f\mapsto  B_{0,1}(f)[\oo(f)]\big],\big[f\mapsto f'B_{1,1}(f)[f,\oo(f)]\big]:H^2(\R)\to H^1(\R)\] 
ensure that $\cO$ is an open subset of $H^2(\R).$

\begin{thm}\label{TK3}
Given $f_0\in \cO$, it holds that  
 \begin{align*} 
  -\p\Phi(f_0)\in \kH(H^2(\R), H^1(\R)).
 \end{align*}
\end{thm}
\begin{proof}
 The proof follows from Theorem \ref{T2}, the relation \eqref{14K**}, Remark~\ref{R:3}, and Lemma~\ref{L:Isom}, by arguing as in the proof of Theorem~\ref{TK2}.
\end{proof}

Finally, we arrive at the  proof our second main result.
\begin{proof}[Proof of Theorem \ref{MT2}] 
Since  for $\sigma=0$ the problem \eqref{P} is equivalent to \eqref{Pest2}, the existence and uniqueness of a maximal solution to \eqref{Pest2}, for each $f_0\in \cO$, follows from \cite[Theorem~8.1.1]{L95}. 
We note that the relation \eqref{reg22} together with Theorem~\ref{TK3} ensures that
all the assumptions of \cite[Theorem 8.1.1]{L95} are satisfied. 
Furthermore,  $(iii)$  follows  from  \cite[Proposition~8.2.3 and Theorem~8.3.9]{L95}, and $(iv)$ from \cite[Proposition~8.2.1]{L95}.

In order to prove $(v)$,  we assume that $f=f(\,\cdot\,; f_0)\in B((0,T), H^{2+\e}(\R))$ for some $T<T_+(f_0)$ and $\e\in(0,1).$
Given $(\lambda_1,\lambda_2)\in(0,\infty)^2$, which we view below as  parameters in an associated nonlinear evolution problem, we define    the function 
\[
f_{\lambda_1,\lambda_2}(t,x):=f(\lambda_1 t,x+\lambda_2 t), \qquad   x\in\R, \, 0\leq t\leq T_{\lambda_1}:=T/\lambda_1,
\]
Since $f\in C([0,T],\cO)\cap C^1([0,T], H^1(\R))\cap C^\e_\e((0,T], H^2(\R))$, cf. Theorem~\ref{MT2}~$(ii)$, our assumption together with the translation invariance of $\cO$ implies that
\begin{align}\label{IQ1}
 f_{\lambda_1,\lambda_2}\in C([0,T_{\lambda_1}],\cO)\cap C^1([0,T_{\lambda_1}], H^1(\R))\cap C^\e_\e((0,T_{\lambda_1}], H^2(\R)). 
\end{align}
Therefore,  the  function $u:=f_{\lambda_1,\lambda_2}$
is a solution to  the nonlinear  evolution problem  
\begin{align}\label{QC}
\p_tu= \Psi(u,\lambda_1,\lambda_2) ,\quad t\geq0,\qquad u(0)=u_0,
\end{align}
for the initial data $u_0:=f_0$, where $ \Psi:\cO\times (0,\infty)^2\subset H^2(\R)\times \R^2\to  H^1(\R)$ denotes the    operator  
\begin{align*}
  \Psi(u,\lambda_1,\lambda_2):=\lambda_1\Phi(u)+\lambda_2\p_x u.
\end{align*}
Recalling \eqref{reg22}, we get $\Psi\in C^\omega(\cO\times (0,\infty)^2,  H^1(\R))$.
Moreover, given $(u_0,\lambda_1,\lambda_2)\in \cO\times (0,\infty)^2$,
the Fr\'echet derivative of $\Psi$ with respect to $u$ is given by
\[
\p_u\Psi(u_0,\lambda_1,\lambda_2)=\lambda_1\p\Phi(u_0)+\lambda_2\p_x.
\]
Since $\lambda_2\p_x$ is a  Fourier multiplier of first order with purely imaginary symbol,
we may revisit the   computations in Theorem~\ref{T2}, Lemma~\ref{L:Isom}, and Theorem~\ref{TK3}, to deduce that   $-\p_u\Psi((u_0,\lambda_1,\lambda_2))$ belongs to $\kH(H^2(\R), H^1(\R))$
for all $(u_0,\lambda_1,\lambda_2)\in \cO\times (0,\infty)^2.$
According to \cite[Theorem~8.1.1 and Theorem~8.3.9]{L95}, the  problem \eqref{QC} possesses for each $(u_0,\lambda_1,\lambda_2)\in \cO\times (0,\infty)^2$ a unique maximal solution $u=u(\,\cdot\,; u_0,\lambda_1,\lambda_2)$ 
(that satisfies similar properties as in Theorem~\ref{MT2}~$(i)-(ii)$), the  set 
$$\Omega:=\{(t,u_0,\lambda_1,\lambda_2)\,:\, (u_0,\lambda_1,\lambda_2)\in \cO\times (0,\infty)^2,\, 0<t<T_+((u_0,\lambda_1,\lambda_2))\}$$
is open and 
\[
[(t,u_0,\lambda_1,\lambda_2)\mapsto u(t; u_0,\lambda_1,\lambda_2)]:\0\to \cO\quad \text{is real-analytic.}
\]
Hence, for our special initial data $f_0$, it follows due to \eqref{IQ1} that $T_+(f_0,\lambda_1,\lambda_2)>T_{\lambda_1}$, and moreover $u(t; u_0,\lambda_1,\lambda_2)=f_{\lambda_1,\lambda_2}(t)$ for all $t\leq T_{\lambda_1}$.
Given  $t_0\in(0,T)$, we choose $\delta>0$ such that  $t_0<T_{\lambda_1}$ for all $(\lambda_1,\lambda_2)$ belonging to the disc $ D_\delta((1,1))$.
Hence, we    conclude  that in particular
\begin{align*}
[(\lambda_1,\lambda_2)\mapsto f_{\lambda_1,\lambda_2}(t_0)]:D_\delta((1,1))\to H^1(\R)
\end{align*}
is  a real-analytic map.
The conclusion is now immediate, see e.g. the proof of \cite[Theorem~1.3]{M16x}.
\end{proof}

 \vspace{0.5cm}
\hspace{-0.5cm}{ \bf Acknowledgements} 
The authors thank the anonymous referee for  the valuable suggestions which have improved the quality of the article.

\bibliographystyle{siam}
\bibliography{B}

\end{document}